\documentclass[a4paper]{article}

\usepackage{a4}
\usepackage{amsmath, amssymb, amsthm}   
\usepackage{graphicx,color}                   
\usepackage{float}  
\usepackage{caption}
\usepackage{subcaption} 
\usepackage{xspace}         
\usepackage[scientific-notation=true]{siunitx}         

\usepackage{tikz}
\usetikzlibrary{positioning} 
\usetikzlibrary{calc}
\usepackage{pgfplots}
\pgfplotsset{width=10cm,compat=1.9}

\definecolor{light_gray}{gray}{0.75}
\definecolor{lighter_gray}{gray}{0.5}
\colorlet{light_blue}{blue!20}
\definecolor{dark_green}{rgb}{0.0, 0.6, 0.0}
\definecolor{royal_blue}{rgb}{0.0, 0.22, 0.66}
\definecolor{salmon}{rgb}{1.0, 0.55, 0.41}
\definecolor{gold}{rgb}{0.8, 0.63, 0.21}
\definecolor{navy_blue}{rgb}{0.0, 0.0, 0.5}

\usepackage{hyperref}
\usepackage[all]{hypcap}

\newcommand{\ep}{\varepsilon}    
\newcommand{\bb}{{\boldsymbol b}} 
\newcommand{\cdr}{\emph{Convection-Diffusion-Reaction}\xspace}

\newcommand{\me}{E\in \mathcal{E}_h}
\newcommand{\bn}{\boldsymbol n} 
\newcommand{\bt}{\boldsymbol t} 
 
\newcommand{\cF}{{\mathcal F}}

\newcommand{\afce}{{\emph{AFC-energy}}\xspace}
\newcommand{\afcse}{{\emph{AFC-SUPG-energy}}\xspace}

\newcommand{\revi}[1]{{\color{black} #1}}

\newcommand{\blist}{\begin{list}{}{\itemsep0.0ex\parsep0.1ex\topsep0.2ex\leftmargin1.6em\labelwidth1.3em}}

\theoremstyle{plain}
\newtheorem{lemma}{Lemma}
\theoremstyle{plain}
\newtheorem{theorem}[lemma]{Theorem}

\theoremstyle{remark}
\newtheorem{remark}[lemma]{Remark}
\theoremstyle{definition}

\title{A Residual Based A Posteriori Error Estimators for AFC Schemes for Convection-Diffusion Equations}
\author{Abhinav Jha
\footnote{RWTH Aachen University, Applied and Computational Mathematics, Schinkelstra\ss e 2, 52062, Aachen, \texttt{jha@acom.rwth-aachen.de}}}
\date{}

\begin{document}
\maketitle
\begin{abstract}
In this work, we propose a residual-based a posteriori error estimator for algebraic flux-corrected (AFC) schemes for stationary convection-diffusion equations. A global upper bound is derived for the error in the energy norm for a general choice of the limiter, which defines the nonlinear stabilization term. In the diffusion-dominated regime, the estimator has the same convergence properties as the true error. A second approach is discussed, where the upper bound is derived in a posteriori way using the Streamline Upwind Petrov Galerkin (SUPG) estimator proposed in \cite{JN13}. Numerical examples study the effectivity index and the adaptive grid refinement for two limiters in two dimensions.
\end{abstract}

\section{Introduction}
In this paper we will study the steady-state \cdr equations given as follows:
\begin{equation}\label{eq:cdr_eqn}
\begin{aligned}
-\varepsilon \Delta u+\boldsymbol{b}\cdot\nabla u+cu&=f&& \mathrm{on}\ \Omega,
\\u&=u_D&& \mathrm{on}\ \Gamma_D,
\\\varepsilon\partial_{\boldsymbol{n}}u&=g&& \mathrm{on}\ \Gamma_N,
\end{aligned}
\end{equation}
where $\varepsilon>0$ is the constant diffusion coefficient, $\bb$ is the convective transport flow with $\nabla \cdot \bb =0$, $c$ is the reaction, $\Omega$ is a polygonal domain in $\mathbb{R}^d,\ d\geq 2$, with Lipschitz boundary $\Gamma$ consisting of two components the Dirichlet boundary, $\Gamma_D$ and the Neumann boundary, $\Gamma_N$, and $u_D$ and $g$ are the Dirichlet and Neumann boundary conditions, respectively. Such equations model the transport of a quantity such as a temperature or concentration. We are interested in the case when convection dominates diffusion as it leads to the formation of layers on the boundary and in the interior of the domain. Hence, one would like a discretization that approximates these layers properly, i.e., they should be sharp and physically consistent, which for \cdr equations means that they satisfy the discrete maximum principle (DMP). In this work, we focus on nonlinear discretizations, namely the algebraic flux correction schemes (AFC) (see \cite{Ku06, Ku07}). The AFC schemes belong to a small class of discretizations that not only compute the layer sharply but also give physically consistent results. The first convergence analysis relating to the AFC schemes has been proposed in \cite{BJK16} using the so-called Kuzmin limiter. The analysis regarding the DMP and convergence of the scheme relies on certain assumptions on the grid. A new definition of the stabilization parameter has been proposed in \cite{BJK17}, called the BJK limiter, which makes the scheme linearity preserving. The first comprehensive study regarding the solvability of the nonlinear problem arising in the AFC scheme has been presented in \cite{JJ18, JJ19} where it has been noted that the nonlinear problem arising for the BJK limiter is more difficult to solve as compared to the Kuzmin limiter.

An approach to approximate the layers properly and reduce the number of unknowns is the use of highly non-equidistant meshes instead of equidistant (or uniform) meshes. One may begin with some uniform mesh, compute a numerical solution on it, and then use information from this to adapt the grid in an a posteriori way, thereby obtaining a grid more suited to the problem. This technique is referred to as \emph{adaptive methods based on a posteriori error estimation}. Modern interest in a  posteriori error estimation for finite element methods (FEMs) for two-point boundary value problems began with the pioneering work of Babu\v{s}ka and Rheinboldt~\cite{BR78}. In the review, \cite{Stynes05} the author prophesizes that adaptive methods will triumph over other methods to solve \cdr equations.

A posteriori error estimation for \emph{Convection-Diffusion-Reaction} equations has received a lot of attention from the past three decades. A review of all the estimators proposed for these equations is beyond this work scope, but some examples of estimators obtained using different techniques can be found in \cite{Ver98, APS05, San08, JN13}. One of the initial studies for the comparison of different estimators using the Streamline Upwind Petrov Galerkin (SUPG) solution of \cdr equations was done in \cite{John00}, and it was shown that none of the estimators was robust with respect to the diffusion coefficient, $\varepsilon$. By robustness, we mean that the equivalence constants between the estimator and the error should be independent of how much convection-dominated the problem is. Work towards deriving a robust estimator was proposed in \cite{Ver05} where the analysis from \cite{Ver98} was extended by adding a dual norm of the convective derivative to the energy norm, but the additional term in the norm can only be approximated. A generalization of the robust estimators was considered in \cite{TV15}, where the analysis was applied to linear stabilized schemes. Robust a posteriori error estimators for $L^1(\Omega)$ and $L^2(\Omega)$ norm of the error can be found in \cite{HGMF06, HFD08, HDF11}. In \cite{JN13} a robust estimator is proposed in the same norm in which the a priori analysis is performed for the SUPG method, namely the SUPG norm. Here the analysis relied on certain hypotheses, including the interpolation of the solution.

One of the drawbacks of all the above-mentioned estimators is the presence of certain constants which can only be approximated. Results related to finding a fully computable upper bound for the error of convection-diffusion equations have gained attention recently and can be found in \cite{AABR13, ESV10}. For the algebraic flux correction schemes (AFC), a fully computable estimator was proposed in \cite{ABR17} with respect to the energy norm. This was the first work where an a posteriori error estimator has been derived for the AFC schemes to the best of our knowledge. It is shown that the estimator is not robust with respect to $\varepsilon$, and also the local efficiency of the scheme relied on certain assumptions, including the Lipschitz continuity of the nonlinear term and the linearity preservation of the scheme.

In this work, we propose a new residual-based estimator for the AFC schemes in the energy norm. Our analysis will consider piecewise linear elements as AFC schemes are restricted to the lowest order element. Results on some concrete choices of constants that appear in certain trace inequalities will be presented. The paper is organized as follows: Sec.~\ref{sec:prelim_post} introduces certain notations, definitions, and auxiliary results that will be used in our a posteriori error analysis. In Sec.~\ref{sec:post_error} a global upper bound and a formal local lower bound are derived for the error in the energy norm. The reason for calling the lower bound a formal lower bound will be made clear later in the paper. Here, we also present another strategy for deriving an upper bound using the SUPG solution. Result obtained with numerical simulations are presented in Sec.~\ref{sec:numres_post} in two dimensions. Main observations include that the proposed residual-based error estimator has, in the diffusion-dominated regime, the same convergence properties as the true error, the actual choice of the limiter plays a minor role in the strategy which uses the SUPG solution, and that the convergence of the AFC scheme with Kuzmin limiter becomes irregular on adaptive grids with red-green refinements (see \cite{Ver13}) once the problem becomes locally diffusion-dominated. Finally, some conclusions and an outlook are given.

\section{Preliminaries}\label{sec:prelim_post}
Throughout this paper we use standard notions for Sobolev spaces and their norms (see \cite{Ada75}). Let $\Omega\subset \mathbb{R}^d$ be a measurable set, then inner product in $L^2(\Omega)$ is denoted by $\left( \cdot, \cdot\right)$. The norm (semi-norm) on $W^{m,p}(\Omega)$ is denoted by $\|\cdot\|_{m,p,\Omega}$ ($|\cdot|_{m,p,\Omega}$), with the convention $\|\cdot\|_{m,\Omega}=\|\cdot\|_{m,2,\Omega}$.

In Eq.~\eqref{eq:cdr_eqn} the Dirichlet part $\Gamma_D$ has a positive $(d-1)$-dimensional Lebesgue measure and $\partial \Omega^-\subset \Gamma_D$, where $\partial \Omega^-$ being the inflow boundary of $\Omega$, i.e., 
$$
\partial \Omega^-=\lbrace x\in \partial \Omega:\ \boldsymbol{b}(x)\cdot \bn(x)<0\rbrace,
$$
where $\bn(x)$ is the outward unit normal. We assume that $\varepsilon\in \mathbb{R}^+,\ \bb\in W^{1,\infty}(\Omega),\ c\in L^{\infty}(\Omega),$ $\revi{f\in L^2(\Omega)}$, and Eq.~\eqref{eq:cdr_eqn} is scaled such that $\|\bb\|_{L^{\infty}(\Omega)}=\mathcal{O}(1)$ and $\|c\|_{L^{\infty}(\Omega)}=\mathcal{O}(1)$. We are interested in the case of convection domination, so we have additional assumption of $0<\varepsilon\ll 1$. 

It is well known that under the assumption
\begin{equation}
\left(c(x)-\frac{1}{2}\nabla \cdot \bb(x)\right)\geq \sigma_0 > 0,
\end{equation}
Eq.~\eqref{eq:cdr_eqn} possesses a unique weak solution $u\in \revi{C(\overline{\Omega})\cap} H_D^1(\Omega)$ that satisfies
\begin{equation}\label{eq:bilinear_form}
a(u,v)=\langle f,v\rangle +\left\langle g,v\right\rangle_{\Gamma_N}\quad \forall v\in H^1_{0,D}(\Omega)
\end{equation}
with 
\begin{equation}\label{eq:bilinear_form_1}
a(u,v)=\varepsilon (\nabla u, \nabla v)+(\boldsymbol{b}\cdot \nabla u, v)+(cu,v),
\end{equation}
$H_D^1(\Omega)=\{v\in H^1(\Omega):\ v|_{\Gamma_D}=u_D\}$, $H_{0,D}^1(\Omega)=\{v\in H^1(\Omega):\ v|_{\Gamma_D}=0\}$, $\langle \cdot, \cdot \rangle$ the duality pairing between $H^1_{0,D}(\Omega)$ and it's dual, \revi{and $\langle \cdot, \cdot \rangle_{\Gamma_N}$ the duality pairing restricted to the Neumann boundary}, e.g. see \cite[Sec.~III.1.1]{RST08}.

The algebraic flux correction scheme for Eq.~\eqref{eq:cdr_eqn} reads as (see \cite{BJK16}): Find $u_h\in W_h(\subseteq C(\overline{\Omega})\cap H^1_D(\Omega))$ such that
\begin{equation}\label{eq:afc_def}
a_{\mathrm{AFC}}(u_h;u_h,v_h)=\langle f,v_h\rangle+\langle g, v_h\rangle_{\Gamma_N}\quad \forall v_h\in V_h\left(\subseteq C(\overline{\Omega})\cap H^1_{0,D}(\Omega)\right),
\end{equation}
with $a_{\mathrm{AFC}}(\cdot,\cdot):H_D^1(\Omega)\times H_{0,D}^1(\Omega)\rightarrow \mathbb{R}$ such that
$$
a_{\mathrm{AFC}}(u_h;u_h,v_h):=a(u_h,v_h)+d_h(u_h;u_h,v_h),
$$
where \revi{$W_h,\ V_h$ are linear finite-dimensional subspaces of $C(\overline{\Omega})\cap H^1_{D}(\Omega)$ and $C(\overline{\Omega})\cap H^1_{0,D}(\Omega)$, respectively, }
\begin{equation}\label{eq:d_h_point_formulation}
d_h(w;u,v)=\sum_{i,j=1}^N\left(1-\alpha_{ij}(w)\right)d_{ij}\left(u(x_j)-u(x_i)\right)v(x_i)\quad \forall u,v,w\in C(\overline{\Omega}),
\end{equation}
$\alpha_{ij}(w)$ are the solution-dependent limiters, $d_{ij}$ is the artificial diffusion matrix \revi{defined by}
$$
\revi{d_{ij}=-\max \lbrace a_{ij}, 0, a_{ji}\rbrace,\ i\neq j,\quad d_{ii}=-\sum_{j=1,j\neq i}^Nd_{ij},}
$$ \revi{$a_{ij}$ the stiffness matrix entries corresponding to Eq.~\eqref{eq:bilinear_form_1},} $N$ the total number of nodes, and $a(u_h,v_h)$ is given by Eq.~\eqref{eq:bilinear_form_1}. For our analysis we will be assuming homogeneous Dirichlet conditions, i.e., $u_D=0$.

In \cite{BJKR18} a different representation of $d_h(\cdot;\cdot,\cdot)$ is given for conforming piecewise linear finite element functions $u$ and $v$, which reads as
\begin{equation}\label{eq:d_h_edge_formulation}
d_h(w;u,v)=\sum_{E\in \mathcal{E}_h}\left(1-\alpha_E(w)\right)|d_E|h_E(\nabla u\cdot \bt_E,\ \nabla v\cdot \bt_E)_E,
\end{equation}
where $\mathcal{E}_h$ is the set of all edges, $\bt_E$ is the tangential unit vector on edge $E$, \revi{and $(\cdot ,\cdot)_E$ is the $L^2$ inner product defined on $E$}. Results regarding the existence and uniqueness (of the linearized system) of the solution can be found in \cite{BJK16}. \revi{We want to note here that by abuse of notation $\alpha_{ij}$ and $\alpha_E$ refer to the same quantities, i.e., the solution-dependent limiters (similarly for $d_{ij}$ and $d_E$). The notation $\alpha_E$ and $d_E$ will be used while referring to the $d_h(\cdot; \cdot, \cdot)$ formulation given by Eq.~\eqref{eq:d_h_edge_formulation}.}

For $u,v,w,u_1,u_2\in C(\overline{\Omega})$ we have the following properties of $d_h(\cdot;\cdot,\cdot )$ (see \cite{BJK16}), 
\begin{enumerate}
\item \emph{Non-negativity}: $0\leq d_h(w;v,v)$.
\item \emph{Linearity}: 
\begin{equation}\label{linearity}
\begin{aligned}
d_h(w;u_1+u_2,v)&=d_h(w;u_1,v)+d_h(w;u_2,v),
\\d_h(w;v,u_1+u_2)&=d_h(w;v,u_1)+d_h(w;v,u_2).
\end{aligned}
\end{equation}
\item \emph{Semi-Norm property, Cauchy-Schwarz inequality}:

\begin{equation}\label{seminorm_prop_d_h}
d_h(w;u,v)\leq d_h^{1/2}(w;u,u)d_h^{1/2}(w;v,v).
\end{equation}
\end{enumerate}

Our a posteriori error estimator will be derived with respect to the energy norm,
\begin{equation}\label{eq:energy_norm}
\|v\|_a^2=\varepsilon |v|^2_{1,\Omega}+\sigma_0 \|v\|_{0,\Omega}^2\ \ \ \forall v\in H^1(\Omega).
\end{equation}

We would also like to mention the induced AFC norm of the system, which is used for its a priori analysis (\cite{BJK16, BJK17}) and which is the starting point of our a posteriori analysis,
\begin{equation}\label{eq:afc_norm}
\|u\|_{\mathrm{AFC}}^2=\|u\|_a^2+d_h(u_h,u,u)\ \ \ \forall u\in H^1(\Omega).
\end{equation}

Let $\lbrace \mathcal{T}_h\rbrace\ (h>0)$ be a family of triangulations consisting of simplices that 
partitions $\Omega$. It will be assumed that the partitions are admissible, i.e.,
any two mesh cells are either disjoint, or share a complete $m$ face, $0\leq m\leq d-1$.
Next, we assume its shape regularity, i.e., there exists a constant $C_{\mathrm{shrg}} > 0$ 
such that for each mesh cell $K\in \mathcal{T}_h$ holds 
\begin{equation}\label{eq:shape_regu_00}
\rho_K \ge C_{\mathrm{shrg}} h_K,
\end{equation}
where $h_K$ and $\rho_K$ denote the diameter of $K$ and the diameter of the largest ball inside $K$, respectively. The characteristic parameter of the triangulation is given by $h=\mathrm{max}_{K\in \mathcal{T}_h}$. We use $|K|$ as a symbol for the volume of a mesh cell $K$. The boundary $\partial K$ of $K$ consists of $m$-dimensional linear manifolds, $0\leq m\leq d-1$, called $m$-faces. The $0$-faces are the vertices of the mesh cell, the $1$-faces are the edges, and the $(d-1)$-faces are called facets or faces. The set of all edges is denoted by $\mathcal E_h$ and the edges of a mesh cell $K$ by $\mathcal E_h(K)$. The set of all faces is denoted by $\cF_h=\cF_{h,\Omega}\cup\cF_{h,D}\cup\cF_{h,N}$, where $\cF_{h,\Omega}$, $\cF_{h,D}$, and $\cF_{h,N}$ denote the interior, Dirichlet, and Neumann faces, respectively. In 2d, it holds that $\mathcal E_h = \cF_h$. The set of mesh cells having a common face $F$ is denoted by $\omega_F=\cup_{F\subset \partial K'}K'$ 
and $\omega_K$ denotes the patch of mesh cells that have a joint face with $K$. 

\begin{remark}[Consequences of the shape regularity assumption \eqref{eq:shape_regu_00}]
We will only discuss the 2d case here, but the result can be extended to 3d. 

Denote the edges of an arbitrary triangle $K$ by $E_1$, $E_2$, and $E_3$,
the angle opposite the edge $E_i$ by $\theta_i$, and the length of $E_i$ by $h_{E_i}$,
$i=1,2,3$. Then, the diameter of the largest ball inside $K$ can be computed by 
$$
\rho_K = \frac{2|K|}{h_{E_1}+h_{E_2}+h_{E_3}}.
$$
Hence, for a given triangulation, one can compute $\rho_K/h_K$ for each mesh cell, such that one gets information on the constant $C_{\mathrm{shrg}}$. Likewise, it is 
$$
\rho_K = \frac{h_{E_1}}{\cot{\frac{\theta_2}2} + \cot{\frac{\theta_3}2}}
$$
and similarly for the other edges. Since $\theta_2>0$, $\theta_3 > 0$, and $\theta_2 + \theta_3 < \pi$, one can check that 
the denominator is larger than $2$ such that $\rho_K < h_{E_1}$ and similarly for the 
two other edges:
\begin{equation}\label{eq:shape_regu}
h_{E_i} > \rho_K, \quad h_{E_i} \ge C_{\mathrm{shrg}} h_K, \quad i=1,2,3.
\end{equation}
In 2d, the shape regularity condition \eqref{eq:shape_regu_00} is 
equivalent with the minimal angle condition, i.e., there is a minimal angle $\theta_0>0$ 
for all triangles and all triangulations from the family of triangulations (see \cite[Pg.~130,~3.1.3]{Cia78}). The minimal angle condition implies a maximal angle condition. Altogether, there is 
a positive constant say, $C_{\mathrm{cos}} < 1$ such that for all $\mathcal T_h$ and all $K \in \mathcal T_h$
\begin{equation}\label{eq:cosine_est}
\cos(\theta_i) \le C_{\mathrm{cos}} \quad i=1,2,3.
\end{equation}
For a given triangulation, $C_{\mathrm{cos}}$ can be 
computed. \revi{This remark would be used later in the computation of certain constants related to the estimate of the trace on the edge cell.}
\end{remark}

\revi{\subsection{Limiters}
To end the preliminaries, we mention the limiters used in the discussion and the numerical simulations.
\subsubsection{Kuzmin Limiter}
This limiter has been proposed in \cite{Ku07}. It is applicable to $\mathbb{P}_1$ and $\mathbb{Q}_1$ elements. The existence and uniqueness (of the linearized) of the solution have been proposed in \cite{BJK16}. The limiters are computed as follows:
\begin{enumerate}
    \item Compute 
    \begin{eqnarray*}
       P_i^+ & = & \sum_{j=1,a_{ji}\leq a_{ij}}^N \max\left\lbrace d_{ij}(u_j-u_i), 0\right\rbrace, \nonumber \\
       P_i^- & = & \sum_{j=1,a_{ij}\leq a_{ji}}^N \min\left\lbrace d_{ij}(u_j-u_i), 0\right\rbrace.\nonumber
    \end{eqnarray*}
    \item Compute 
       \begin{equation*}
       Q_i^+ = \sum_{j=1}^N \min\left\lbrace d_{ij}(u_j-u_i), 0\right\rbrace, \qquad Q_i^- = -\sum_{j=1}^N \max\left\lbrace d_{ij}(u_j-u_i), 0\right\rbrace.
    \end{equation*}
    \item Compute
    \begin{equation*}
       R_i^+ = \min\left\{ 1,\frac{Q_i^+}{P_i^+} \right\}, \qquad R_i^+ = \min\left\{ 1,\frac{Q_i^-}{P_i^-} \right\},\qquad i=1,\dots, M,
    \end{equation*}
    where $M$ are the number of non-Dirichlet degrees of freedoms. If the \(P_i^+\) or \(P_i^-\) is zero, we set \(R_i^+=1\) or \(R_i^-=1\), respectively.  \(R_i^+\) and \(R_i^-\) are set to $1$ for Dirichlet nodes as well. 
    \item Compute
    \[ \alpha_{ij} = \begin{cases} 
        R_i^+ & \mbox{ if } d_{ij}(u_j-u_i)>0,\\
        1 & \mbox{ if } d_{ij}(u_j-u_i)=0,\\
        R_i^- & \mbox{ if } d_{ij}(u_j-u_i)<0,
      \end{cases}
   \]
\end{enumerate}
for $i,j=1,\dots,N$.

\subsubsection{BJK Limiter}
This limiter has been proposed in \cite{BJK17} which makes the AFC scheme linearity preserving. This limiter is only applicable to $\mathbb{P}_1$ elements. For a detailed review of the limiter, we refer to \cite{BJK17}. The limiters are computed as follows:
\begin{enumerate}
    \item Compute 
    \begin{eqnarray*}
       P_i^+ & = &\sum_{j\in N_i\cup \{i\}}^N \max\left\lbrace d_{ij}(u_j-u_i), 0\right\rbrace,\nonumber \\
       P_i^- & = &\sum_{j\in N_i\cup \{i\}}^N \min\left\lbrace d_{ij}(u_j-u_i), 0\right\rbrace,\nonumber
    \end{eqnarray*}
    where $N_i$ is the set of nodes for which there is an entry in the stiffness matrix's sparsity pattern, i.e., $N_i$ is the set of all neighbor degrees of freedom of $x_i$ including $x_i$.
    \item Compute 
       \begin{equation*}
       Q_i^+ = q_i\left(u_i-u_i^{\max}\right), \qquad Q_i^- = q_i\left(u_i-u_i^{\min}\right),
    \end{equation*}
    where 
    \begin{eqnarray*}
    u_i^{\max}  & = & \max_{j\in N_i}u_j,\nonumber \\
    u_i^{\min}  & = &\min_{j\in N_i}u_j,\nonumber \\
    q_i & = &\sum_{j\in N_i}\gamma_i d_{ij},\nonumber
    \end{eqnarray*}
    and $\gamma_i$ is a positive constant computed for interior nodes as given in \cite[Rem.~6.2]{BJK17}.
    \item Compute
    \begin{equation*}
       R_i^+ = \min\left\{ 1,\frac{Q_i^+}{P_i^+} \right\}, \qquad R_i^+ = \min\left\{ 1,\frac{Q_i^-}{P_i^-} \right\},\quad i=1,\dots, M.
    \end{equation*}
If the \(P_i^+\) or \(P_i^-\) is zero, we set \(R_i^+=1\) or \(R_i^-=1\), respectively. \(R_i^+\) and \(R_i^-\) are set to $1$ for Dirichlet nodes as well. 
    \item Compute
    \[ \overline{\alpha}_{ij} = \begin{cases} 
        R_i^+ & \mbox{ if } d_{ij}(u_j-u_i)>0,\\
        1  & \mbox{ if } d_{ij}(u_j-u_i)=0, \\
        R_i^- & \mbox{ if } d_{ij}(u_j-u_i)<0,
      \end{cases}\qquad i=1,\dots, M,\ j=1,\dots, N.
   \]
\end{enumerate}
Finally, one sets
\begin{eqnarray*}
\alpha_{ij} = \min \left\lbrace\overline{\alpha}_{ij}, \overline{\alpha}_{jj}\right\rbrace,\qquad && i,j=1,\dots,M, \nonumber \\
\alpha_{ij} =  \overline{\alpha}_{ij},\qquad && i=1,\dots,M,\ j=M+1,\dots, N. \nonumber
\end{eqnarray*}}

\subsection{Auxiliary Results}
In this subsection, we would mention certain standard results used for a posteriori error estimation. We would also give some concrete choices of constants in certain trace results. We will assume that the triangulations are regular.

\begin{lemma} \emph{(\textbf{Inverse estimate}) (\cite[Lemma~4.5.3]{BS08})}
Let $C_{\mathrm{shrg}} h\leq h_K\leq h$, where $0<h\leq 1$, and $\mathcal{P}_h$ be a polynomial subspace of $H^m(K)$. Then for $0\leq l\leq m$ there exists a constant $C_{\mathrm{inv}}$ such that for all $v\in \mathcal{P}_h$ and $K\in \mathcal{T}_h$, we have
\begin{equation}\label{eq:inv_est}
\| v_h\|_{m,K}\leq C_{\mathrm{inv}}h_K^{l-m}\|v_h\|_{l,K}.
\end{equation}
\end{lemma}

\begin{theorem} \emph{(\textbf{Interpolation estimate}) (\cite[Corollary~4.8.15]{BS08})} Let $q\in [1,\infty]$ and $s\leq t\leq 1$. Let, $I_h:W^{t,q}(\Omega)\rightarrow V_h$ denote a bounded linear interpolation operator. Then, it satisfies $\forall v\in W^{t,q}(\Omega)$ and all mesh cells $K\in \mathcal{T}_h$
\begin{equation}\label{eq:interpolation_estimate}
\revi{\left(\sum_{K\in \mathcal{T}_h}\|v-I_hv\|_{s,q,K}^q \right)^{1/q}\leq C_Ih^{t-s}|v|_{t,q,\Omega}}
\end{equation}
\end{theorem}

\begin{remark}\label{rem:interpoaltion_assumption}
For the analysis, we need a stable quasi-interpolation operator, which is identity on the finite element space, i.e.,
$$
I_hu_h=u_h\quad \forall\ u_h\in V_h.
$$
\revi{One candidate for such an interpolation is the Scott-Zhang interpolation operator, (see \cite{SZ90}) which will be used in this paper}. It is important to note that $I_h$ cannot be the nodal interpolation operator as it is not $L^2$-stable and $L^2$-stability is required further in the proof.
\end{remark}

\begin{remark}
For $s=t$ in Eq.~\eqref{eq:interpolation_estimate}, one gets with $u_h = I_hu_h$
\begin{eqnarray}\label{eq:stab_inter}
\sum_{K\in T_h}\|u-I_hu\|_{s,q,K}^q &\le& \sum_{K\in T_h} \left(\|u-u_h\|_{s,q,K}^q + \|I_hu - I_hu_h\|_{s,q,K}^q\right) \nonumber \\
&=& (1+C_I)  \|u-u_h\|_{s,q,\Omega}^q.
\end{eqnarray}
\end{remark}

A trace inequality which relates the $L^2(F)$ norm on a face of a mesh cell $K$
to norms defined on $K$ was proved in \cite{Ver98}.

\begin{lemma} \emph{(\cite[Lemma~3.1]{Ver98})}
Let $v\in H^1(K)$ and $F \subset  \partial K$ with diameter $h_F$, then it holds
\begin{equation}\label{eq:trace_inequality_F}
\|v\|_{0,F}\le C \left(h_F^{-1/2}\|v\|_{0,K}+\|v\|_{0,K}^{1/2}\|\nabla v\|_{0,K}^{1/2}\right).
\end{equation}
\end{lemma}

\begin{lemma} Let E be an edge with length $h_E$ and $v$ be a linear function on E, then
\begin{equation}\label{eq:lemma_1}
\|\nabla v\cdot \bt_E\|_{0,E}^2\leq \|\nabla v\|_{0,E}^2.
\end{equation}
\end{lemma}
\begin{proof}
We know that $\|\bt_E\|_{\revi{\infty},E}=1$. Hence, using this, we get
$$
\|\nabla v\cdot \bt_E\|_{0,E}^2\leq \|\nabla v\|_{0,E}^2\|\bt_E\|_{\revi{\infty},E}^2=\|\nabla v\|_{0,E}^2.
$$
\end{proof}

\begin{lemma}[Estimate of the trace on an edge by the norm on the mesh cell] Let $K \in \mathcal T$ be a 
mesh cell, $\mathcal{E}_h(K)$ the set of all edges of $K$, and $\varphi_h \in \mathbb{P}_1(K)$ be a nodal functional. Then, there exist a constant $C_{\mathrm{edge}}$ independent of $K$ such that
\begin{equation}\label{eq:trace_inequality_E}
\sum_{E\in \mathcal E_h(K)} \|\nabla \varphi_h\cdot \bt_E\|_{0,E}^2 \le C_{\mathrm{edge}}
h_K^{1-d} \|\nabla \varphi_h\|_{0,K}^2.
\end{equation}
\end{lemma}

\begin{proof} The principal way for proving the statement of the lemma is the same for two and three
dimensions. It uses the mapping to the reference cell.
We will present proof for $d=2$.

{\em Relating the norms on $E$ and $\hat E$.} This step is just a one-dimensional consideration
for an edge. Thus, one has to do the same calculations in 2d and 3d.

Let $\hat K$ be the reference triangle with the vertices $\hat V_0 = (0,0)$, 
$\hat V_1 = (1,0)$, and  $\hat V_2 = (0,1)$. Since an additive constant does not play any role, 
it will be assumed that for $\hat{\varphi}_h\in \mathbb{P}_1(\hat{K})$, $\hat\varphi_h(\hat V_0) = 0$,  $\hat\varphi_h(\hat V_1) = \alpha$,  and 
$\hat\varphi_h(\hat V_2) = \beta$ with $\alpha, \beta \in \mathbb R$. Consequently, it is 
$\nabla \hat\varphi_h = (\alpha, \beta)^T$. One obtains for $\hat E = \overline{\hat V_0\hat V_1}$
and $h_{\hat E}= |\hat E| =1$
\begin{equation}\label{eq:trace_inequality_E_00}
\int_{\hat V_0}^{\hat V_1} (\nabla \hat\varphi_h \cdot \bt_{\hat E})^2\ ds = 
\left(\frac{(\hat\varphi_h(\hat V_1)-\hat\varphi_h(\hat V_0))^2}{h_{\hat E}^2} \right) h_{\hat E}
= \alpha^2.
\end{equation}
Analogously, one finds 
\begin{equation}\label{eq:trace_inequality_E_01}
\int_{\hat V_0}^{\hat V_2} (\nabla \hat\varphi_h \cdot \bt_{\hat E})^2\ ds = \beta^2, \qquad 
\int_{\hat V_0}^{\hat V_2} (\nabla \hat\varphi_h \cdot \bt_{\hat E})^2\ ds = \frac1{\sqrt2} 
(\alpha-\beta)^2.
\end{equation}

Let the reference map $F_K\ : \ \hat K \to K$ map $\hat V_0$ to $V_0$ and $\hat V_1$ to $V_1$, 
where $V_0$ and $V_1$ are vertices of $K$. Then it holds that
$\hat\varphi_h(\hat V_0)= \varphi_h(V_0)$ and
$\hat\varphi_h(\hat V_1)= \varphi_h(V_1)$.
Denote $E= \overline{V_0V_1}$, then it is 
\begin{equation*}
\int_{V_0}^{V_1} (\nabla\varphi_h\cdot \bt_{E})^2\ ds = \left(\frac{(\varphi_h(V_0) - \varphi_h(V_1))^2}{h_E^2} \right) h_E.
\end{equation*}
The value of this integral has to be equal to Eq.~\eqref{eq:trace_inequality_E_00}, from what follows that 
$$
\|\nabla\varphi_h\cdot \bt_{E}\|_{0,E}^2 = \frac{h_{\hat E}}{h_E} \|\nabla\hat \varphi_h\cdot \bt_{\hat E}\|_{0, \hat E}^2.
$$
Performing the same considerations for the other two edges, one obtains with Eq.~\eqref{eq:trace_inequality_E_01}
\begin{equation}\label{eq:trace_inequality_E_03}
\|\nabla\varphi_h\cdot \bt_{E}\|_{0,E}^2 \le \frac{\sqrt2}{h_E} \|\nabla\hat \varphi_h\cdot \bt_{\hat E}\|_{0,\hat E}^2.
\end{equation}

{\em 2d: Estimate on the reference cell.} Using Eq.~\eqref{eq:trace_inequality_E_00}, Eq.~\eqref{eq:trace_inequality_E_01} and Young's inequality yields
\begin{eqnarray*}
\sum_{\hat E \subset \partial \hat K} \|\nabla\hat \varphi_h\cdot \bt_{\hat E}\|_{0,\hat E}^2
& = & \alpha^2 + \beta^2 + \frac1{\sqrt2} (\alpha-\beta)^2\\
& \le & \left(1+\sqrt2\right) (\alpha^2+\beta^2).
\end{eqnarray*}
Since 
\begin{equation}\label{eq:trace_inequality_E_04}
\int_{\hat K} (\nabla \hat\varphi_h \cdot \nabla \hat\varphi_h)\ ds = \frac12  (\alpha^2+\beta^2),
\end{equation}
one obtains 
\begin{equation}\label{eq:trace_inequality_E_05}
\sum_{\hat E \subset \partial \hat K} \|\nabla\hat \varphi_h\cdot \bt_{\hat E}\|_{0,\hat E}^2
\le 2 \left(1+\sqrt2\right) \|\nabla \hat\varphi_h\|_{0,\hat K}^2.
\end{equation}

{\em Relating the norms on $\hat K$ and $K$.} From the standard numerical analysis it is known 
that there is a constant $C$ which is independent of $K$, such that 
\begin{equation}\label{eq:trace_inequality_E_16}
\|\nabla \hat\varphi_h\|_{0,\hat K}^2 \le Ch_K^{2-d} \|\nabla \varphi_h\|_{0,K}^2.
\end{equation}

Estimate Eq.~\eqref{eq:trace_inequality_E} is now obtained by combining Eq.~\eqref{eq:trace_inequality_E_03}, Eq.~\eqref{eq:trace_inequality_E_05}, and Eq.~\eqref{eq:trace_inequality_E_16}, and using the shape regularity of the mesh cell Eq.~\eqref{eq:shape_regu}. 
\end{proof}

\begin{remark}[More detailed estimate in 2d]
Let $\varphi_h$ be a linear function on $K$ with $\varphi_h(V_0) = 0$, $\varphi_h(V_1) = \alpha$,
and $\varphi_h(V_2) = \beta$, and $(x_0,y_0)$, $(x_1,y_1)$, and $(x_2,y_2)$ be the coordinates of $V_0,\ V_1,$ and $V_2$ respectively.  Then the standard Hessian form of the plane on $K$ is given by
$$
\varphi_h=-\left(a_4+\frac{a_1x}{a_3}+\frac{a_2y}{a_3}\right),
$$
where $a_1=(y_1-y_0)\beta-(y_2-y_0)\alpha$, $a_2=(x_2-y_0)\alpha-(x_1-x_0)\beta$, $a_3=(x_1-x_0)(y_2-y_0)-(x_2-x_0)(y_1-y_0)$, and $a_4$ is a constant which can be computed by a point on the plane. Now
$$
\nabla \varphi_h=-\frac{1}{a_3}
\begin{pmatrix}
a_1\\
a_2
\end{pmatrix}=-\frac{1}{2|K|}
\begin{pmatrix}
a_1\\
a_2
\end{pmatrix}.
$$
A direct calculation gives that
$$
\nabla\varphi_h \cdot \nabla\varphi_h = 
\frac1{4|K|^2} \left(\alpha^2 h_{E_2}^2 + \beta^2 h_{E_1}^2 - 2 \alpha \beta h_{E_1}
h_{E_2} \cos(\theta_0)\right),
$$
where $E_1$ and $E_2$ are the edges joining $(x_0,y_0)$ with $(x_1,y_1)$ and $(x_2,y_2)$, respectively and $\theta_0$ is the angle between the two edges.

Using the condition Eq.~\eqref{eq:cosine_est} on the maximal cosine, Young's inequality, the shape regularity Eq.~\eqref{eq:shape_regu}, and Eq.~\eqref{eq:trace_inequality_E_04} yields 
\begin{eqnarray*}
\|\nabla\varphi_h\|_{0,K}^2 & \ge &  \frac1{4|K|} \left(\alpha^2 h_{E_2}^2 + \beta^2 h_{E_1}^2 - 2 C_{\mathrm{cos}} |\alpha| |\beta| h_{E_1}
h_{E_2} \right)\\
& \ge & \frac1{4|K|} \left(\alpha^2 h_{E_2}^2 (1-C_{\mathrm{cos}}) + \beta^2 h_{E_1}^2 (1-C_{\mathrm{cos}}) \right)\\
& \ge & \frac{1-C_{\mathrm{cos}}}{4|K|} \rho_K^2 \left(\alpha^2 + \beta^2\right) \\
& = & \frac{1-C_{\mathrm{cos}}}{2|K|} \rho_K^2 \|\nabla \hat\varphi_h\|_{0,\hat K}^2.
\end{eqnarray*}
Combining this estimate with Eq.~\eqref{eq:trace_inequality_E_03}, Eq.~\eqref{eq:shape_regu}, and  Eq.~\eqref{eq:trace_inequality_E_05} leads to 
\begin{eqnarray*}
\sum_{E\in \mathcal E_h(K)} \|\nabla \varphi_h\cdot \bt_E\|_{0,E}^2 & \le &
\frac{\sqrt{2}}{\rho_K} \sum_{\hat E \subset \partial \hat K} \|\nabla\hat \varphi_h\cdot \bt_{\hat E}\|_{0,\hat E}^2\\
& \le & \frac{2\sqrt{2}\left(1+\sqrt2\right)}{\rho_K}  \|\nabla \hat\varphi_h\|_{0,\hat K}^2\\
&\le & \frac{4\sqrt{2}\left(1+\sqrt2\right)|K|}
{(1-C_{\mathrm{cos}})\rho_K^3} 
\|\nabla\varphi_h\|_{0,K}^2.
\end{eqnarray*}
The first factor on the right-hand side scales like $h_K^{-1}$ since $\rho_K \sim h_K$ and 
$|K| \sim h_K^2$. For a given triangulation, 
it is computable. 
\end{remark}

\section{A Posteriori Error Estimator}\label{sec:post_error}
In this section, we propose a new residual-based a posteriori error estimator for the AFC schemes in the energy norm. To the best of our knowledge, only one work has been done in the context of a posteriori error estimation and the AFC schemes (see \cite{ABR17}). A fully computable upper bound has been derived under certain assumptions on the nonlinear stabilization term. In this work, ideas from \cite{AABR13} have been extended to the AFC schemes. The estimator's design relies on introducing certain first-order consistent equilibrated fluxes and then solving a local Neumann problem to get explicit bounds. To show the estimator's local efficiency, two assumptions are made on the nonlinear stabilization ($d_h(\cdot;\cdot,\cdot)$), namely the local Lipschitz continuity and the linearity preservation. Because of the last assumption, this estimator was not applicable to the Kuzmin limiter.

The derivation of an estimator presented in this section follows the standard residual-based approach. We start with the variational formulation and use standard interpolation estimates to bound the terms. We also propose an estimator later in this section that uses the SUPG solution for bounding the error.
\subsection{Residual-Based Estimator}
\subsubsection{Global Upper Bound}\label{sec:upper_bound}
This section will present a global upper bound for the AFC scheme in the energy norm given by Eq.~\eqref{eq:energy_norm}.

Let $u\in H_D^1(\Omega)\cap C(\overline{\Omega})$ be a solution of Eq.~\eqref{eq:bilinear_form} and $u_h\in W_h$ be a solution for Eq.~\eqref{eq:afc_def}, then for $v_h\in V_h$ one obtains with Eq.~\eqref{eq:bilinear_form} and Eq.~\eqref{eq:afc_def}
\begin{eqnarray}\label{eq:galerkin_ortho}
a_{\mathrm{AFC}}(u_h;u-u_h,v_h)
&=&a(u-u_h,v_h)+d_h(u_h;u-u_h,v_h)\nonumber
\\&=&\langle f,v_h\rangle +\langle g,v_h\rangle_{\Gamma_N}-\langle f,v_h\rangle-\langle g,v_h\rangle_{\Gamma_N}+d_h(u_h;u,v_h)
\nonumber
\\&=&d_h(u_h;u,v_h).
\end{eqnarray}
For any $v\in H^1_0(\Omega)\cap C(\overline{\Omega})$, the application of  Eq.~\eqref{eq:afc_def}, Eq.~\eqref{eq:d_h_point_formulation}, and Eq.~\eqref{eq:galerkin_ortho} yields 
\begin{eqnarray*}
\lefteqn{
a_{\mathrm{AFC}}(u_h;u-u_h,v)}\\&=&a_{\mathrm{AFC}}(u_h;u-u_h,v-I_hv)+a_{\mathrm{AFC}}(u_h;u-u_h,I_hv)\\
&= &a(u-u_h,v-I_hv)+d_h(u_h;u-u_h,v-I_hv) +d_h(u_h;u,I_hv)\\
&= &\langle f,v-I_hv\rangle +\langle g,v-I_hv\rangle_{\Gamma_N}+d_h(u_h;u-u_h,v-I_hv)\\ && +d_h(u_h;u,I_hv)-a(u_h,v-I_hv).
\end{eqnarray*}
Taking $v=u-u_h$ in this equation, using $u_h=I_hu_h$, and applying integration by parts, one gets
\begin{eqnarray}\label{eq:erroR_Fqn}
\lefteqn{\|u-u_h\|_{\mathrm{AFC}}^2}\nonumber\\
&=& \|u-u_h\|_a^2+d_h(u_h;u-u_h,u-u_h)\nonumber\\
&=& a_{\mathrm{AFC}}(u_h;u-u_h,u-u_h)\nonumber\\
&=& \langle f,u-I_hu\rangle +\langle g,u-I_hu\rangle_{\Gamma_N}+d_h\left(u_h;u-u_h,u-u_h-I_h(u-u_h)\right)\nonumber \\
&&+d_h(u_h;u,I_hu- I_hu_h) -a(u_h,u-I_hu)\\
&=&\sum_{K\in \mathcal{T}_h}\left( R_K(u_h),u-I_hu\right)_K+\sum_{F\in \cF_h}\langle R_F(u_h),u-I_hu\rangle_F \nonumber \\
&& +d_h\left(u_h;u,I_hu-u_h)+d_h(u_h;u-u_h,u-u_h-I_h(u-u_h)\right),\nonumber
\end{eqnarray}
with
\begin{eqnarray*}
R_K(u_h)&:=&f+\varepsilon \Delta u_h-\boldsymbol{b}\cdot \nabla u_h-cu_h|_K,\nonumber \\
R_F(u_h)&:=&
\left\lbrace
\begin{array}{lc}
-\varepsilon [|\nabla u_h \cdot \bn_F |]_F &  \mathrm{if}\ F\in \cF_{h,\Omega},\\
g-\varepsilon (\nabla u_h \cdot \bn_F)&
\mathrm{if}\ F\in \cF_{h,N},\\
0 & \mathrm{if}\ F\in \cF_{h,D},
\end{array}\right.\nonumber
\end{eqnarray*}
where $[|\cdot|]_F$ denotes the jump across the face $F$.

The terms on the right-hand side of Eq.~\eqref{eq:erroR_Fqn} have to be bounded. For a nodal interpolation operator, the last term in Eq.~\eqref{eq:erroR_Fqn} vanishes, and hence one has to use a quasi-interpolation operator.

For the first term in Eq.~\eqref{eq:erroR_Fqn}, using the Cauchy--Schwarz inequality,
$u_h=I_hu_h$,
the interpolation estimate Eq.~\eqref{eq:interpolation_estimate} with $s=0,\ t=0$,
and the generalized Young's inequality gives
\begin{eqnarray}\label{eq:l_2_norm_est}
\sum_{K\in \mathcal{T}_h} \left( R_K(u_h), u-I_hu\right)_K&\leq& \sum_{K\in \mathcal{T}_h}\|R_K(u_h)\|_{0,K}\|u-I_hu\|_{0,K}\nonumber\\
& = & \sum_{K\in \mathcal{T}_h}\|R_K(u_h)\|_{0,K}\|(u-u_h)-I_h(u-u_h)\|_{0,K}\nonumber\\
&\leq& \sum_{K\in \mathcal{T}_h}\|R_K(u_h)\|_{0,K}C_I\|u-u_h\|_{0,K}\\
& \leq& \frac{C_YC_I^2}{2\sigma_0}\sum_{K\in \mathcal{T}_h}\|R_K(u_h)\|_{0,K}^2+\frac{\sigma_0}{2C_Y}\|u-u_h\|_{0,\Omega}^2,\nonumber
\end{eqnarray}
where $C_Y$ is the Young's inequality constant.

One can also approximate the interpolation error with Eq.~\eqref{eq:interpolation_estimate} and $s=0,\ t=1$,
leading to 
\begin{eqnarray}\label{eq:h_1_norm_est}
\sum_{K\in \mathcal{T}_h} \left( R_K(u_h), u-I_hu\right)_K&\leq& \sum_{K\in \mathcal{T}_h}\|R_K(u_h)\|_{0,K}\|u-I_hu\|_{0,K}\nonumber
\\&\leq &\sum_{K\in \mathcal{T}_h}\|R_K(u_h)\|_{0,K}C_Ih_K|u-u_h|_{1,K} 
\\& \leq &\frac{C_YC_I^2h_K^2}{2\varepsilon}  \sum_{K\in \mathcal{T}_h}\|R_K(u_h)\|_{0,K}^2 \nonumber
\\&&+\frac{\varepsilon}{2C_Y}|u-u_h|_{1,\Omega}^2.\nonumber
\end{eqnarray}
Hence, combining Eq.~\eqref{eq:l_2_norm_est} and Eq.~\eqref{eq:h_1_norm_est} gives
\begin{eqnarray}\label{eq:first_bound}
\lefteqn{\sum_{K\in \mathcal{T}_h} \left( R_K(u_h), u-I_hu\right)_K}\nonumber\\
&\leq&
\frac{C_Y}{2}\sum_{K\in \mathcal{T}_h}\mathrm{min}\left\lbrace \frac{C_I^2}{\sigma_0},\ \frac{C_I^2h_K^2}{\varepsilon}\right\rbrace\|R_K(u_h)\|_{0,K}^2+\frac{1}{2C_Y}\|u-u_h\|_a^2.
\end{eqnarray}

The estimate of the second term in Eq.~\eqref{eq:erroR_Fqn} starts also with the 
Cauchy--Schwarz inequality and using $u_h = I_h u_h$
\begin{eqnarray*}
\sum_{F\in \cF_h}\langle R_F(u_h),u-I_hu\rangle_F &\le& 
\sum_{F\in \cF_h} \| R_F(u_h)\|_{0,F} \|u-I_hu\|_{0,F}\\
& = & \sum_{F\in \cF_h} \| R_F(u_h)\|_{0,F} \|(u-u_h) - I_h(u-u_h)\|_{0,F}.
\end{eqnarray*}
The local trace estimate Eq.~\eqref{eq:trace_inequality_F} is applied to the second factor on the right-hand side. After this, one proceeds essentially as for the mesh cell residual by using the interpolation estimate Eq.~\eqref{eq:interpolation_estimate}, considering the cases $s=t=0$ and $s=0, t=1$ for the interpolation error in $L^2(K)$, 
performing some straightforward calculations, compare \cite{JN13}, and using the shape regularity of the mesh cell, to find 
\begin{equation*}
\|(u-u_h) - I_h(u-u_h)\|_{0,F} \le C_F \min \left\{\frac{h_F^{1/2}}{\varepsilon^{1/2}}, \frac1{\sigma_0^{1/4}\varepsilon^{1/4}} \right\} \|u-u_h\|_a,
\end{equation*}
where the constant $C_F$ depends on the constant from Eq.~\eqref{eq:trace_inequality_F} and the interpolation constant. Applying now the generalized Young's inequality, one gets for the face residuals
\begin{eqnarray}\label{eq:est_face_res}
\lefteqn{
\sum_{F\in \cF_h}\langle R_F(u_h),u-I_hu\rangle_F}\nonumber\\
&\le& \frac{C_Y}{2}\sum_{F\in \cF_h} \min \left\{\frac{C_F^2 h_F}{\varepsilon}, \frac{C_F^2}{\sigma_0^{1/2}\varepsilon^{1/2}} \right\}\| R_F(u_h)\|_{0,F}^2
+\frac{1}{2C_Y}\|u-u_h\|_a^2.
\end{eqnarray}

As intermediate result, one obtains from Eq.~\eqref{eq:erroR_Fqn}, Eq.~\eqref{eq:first_bound}, and Eq.~\eqref{eq:est_face_res}
\begin{eqnarray}\label{eq:est_all_00}
\lefteqn{\|u-u_h\|_a^2+\frac{C_Y}{C_Y-1}d_h(u_h;u-u_h,u-u_h)}\nonumber\\
& \le &\frac{C^2_Y}{2(C_Y-1)}\sum_{K\in \mathcal{T}_h}\mathrm{min}\left\lbrace \frac{C_I^2}{\sigma_0},\ \frac{C_I^2h_K^2}{\varepsilon}\right\rbrace\|R_K(u_h)\|_{0,K}^2\nonumber\\
&& + \frac{C^2_Y}{2(C_Y-1)}\sum_{F\in \cF_h} \min \left\{\frac{C_F^2 h_F}{\varepsilon}, \frac{C_F^2}{\sigma_0^{1/2}\varepsilon^{1/2}} \right\}\| R_F(u_h)\|_{0,F}^2\\
&& + \frac{C_Y}{C_Y-1}d_h(u_h;u,I_hu-u_h) + \frac{C_Y}{C_Y-1}d_h\left(u_h;u-u_h,u-u_h-I_h(u-u_h)\right).\nonumber
\end{eqnarray}

We estimate the last two term in Eq.~\eqref{eq:est_all_00}, by using Eq.~\eqref{linearity}, and Remark~\ref{rem:interpoaltion_assumption}, leading to 
\begin{eqnarray}\label{eq:d_h_est}
\lefteqn{d_h(u_h;u-u_h,u-u_h-I_h(u-u_h))+d_h(u_h;u,I_h(u-u_h))}\nonumber \\
& = & d_h(u_h;u-u_h,u-u_h)-d_h(u_h;u,I_h(u-u_h))\nonumber \\
&& +d_h(u_h;u_h,I_h(u-u_h))+d_h(u_h;u,I_h(u-u_h))\nonumber \\
& = & d_h(u_h;u-u_h,u-u_h)+d_h(u_h;u_h,I_h(u-u_h)).
\end{eqnarray}
Inserting this relation in Eq.~\eqref{eq:est_all_00} reveals that the stabilization term on the left-hand side cancels with the first term on the right-hand side of Eq.~\eqref{eq:d_h_est}.
Consequently, only the energy norm is left to be estimated. 

Since $I_h u-u_h$ is linear on each edge, the second term on the right-hand side of Eq.~\eqref{eq:d_h_est} can be rewritten as integral over the edges, see Eq.~\eqref{eq:d_h_edge_formulation}, and estimated with the Cauchy--Schwarz inequality and the generalized Young's inequality
\begin{eqnarray}\label{eq:est_dh_00}
\lefteqn{
 d_h(u_h;u_h,I_h u-u_h)}\nonumber\\
 & = & \sum_{\me}(1-\alpha_E)|d_E| h_E 
 (\nabla u_h \cdot\bt_E, \nabla (I_h u-u_h)\cdot\bt_E)_E\nonumber\\
 & \le & \sum_{\me}(1-\alpha_E)|d_E| h_E \|\nabla u_h \cdot\bt_E\|_{0,E}
 \|\nabla (I_h u-u_h)\cdot\bt_E\|_{0,E}\nonumber\\
 & \le & \frac{1}{2C_Y\kappa_1}\sum_{\me} \varepsilon h_E^{d-1} \|\nabla (I_h u-u_h)\cdot\bt_E\|_{0,E}^2\nonumber \\
 && + \frac{C_Y\kappa_1}2
 \sum_{\me}\varepsilon^{-1} (1-\alpha_E)^2|d_E|^2 h_E^{3-d} \|\nabla u_h \cdot\bt_E\|_{0,E}^2.
\end{eqnarray} 
The parameter $\kappa_1$ will be defined later. 
The second term is computable. 

Consider the first term in Eq.~\eqref{eq:est_dh_00}. Denoting
$$
C_{\mathrm{edge,max}} = \max_{K\in \mathcal T_h} C_{\mathrm{edge}},
$$
using $h_E \le h_K$, $d-1 > 0$, Eq.~\eqref{eq:trace_inequality_E}, the triangle inequality, and Eq.~\eqref{eq:stab_inter} yields 
\begin{eqnarray}\label{eq:est_dh_01}
\lefteqn{
\frac1\kappa_1\sum_{\me} \varepsilon h_E^{d-1} \|\nabla (I_h u-u_h)\cdot\bt_E\|_{0,E}^2}\nonumber\\
&\le & \frac{\varepsilon}\kappa_1\sum_{K \in \mathcal T_h}\left( \sum_{E\in \partial K}
h_E^{d-1} \|\nabla (I_h u-u_h)\cdot\bt_E\|_{0,E}^2\right)\nonumber\\
& \le & \frac{\varepsilon}{\kappa_1}\sum_{K \in \mathcal T_h} C_{\mathrm{edge}} 
\|\nabla (I_h u-u_h)\|_{0,K}^2\nonumber\\
& \le & \frac{2 \varepsilon C_{\mathrm{edge,max}}}{\kappa_1}\sum_{K \in \mathcal T_h}
\left(\|\nabla (u-u_h)\|_{0,K}^2+ \|\nabla (u- I_h u)\|_{0,K}^2\right)\nonumber\\
& \le & \frac{2 C_{\mathrm{edge,max}} (1+(1+C_I)^2)}{\kappa_1} \|u-u_h\|_a^2.
\end{eqnarray}
Choosing 
\begin{equation}\label{eq:kappa_1}
\kappa_1 =  C_{\mathrm{edge,max}} \left(1+(1+C_I)^2\right),
\end{equation}
then this term multiplied with $(2C_Y)^{-1}$ can be absorbed in the left-hand side of Eq.~\eqref{eq:est_all_00}.

An alternative estimate proceeds similarly to Eq.~\eqref{eq:est_dh_00}
\begin{eqnarray}\label{eq:est_dh_02}
 d_h(u_h;u_h,I_h u-u_h)
  & \le & \frac{1}{2C_Y\kappa_2}\sum_{\me} \sigma_0 h_E^{d+1} \|\nabla (I_h u-u_h)\cdot\bt_E\|_{0,E}^2\nonumber\nonumber \\
 && + \frac{C_Y\kappa_2}2
 \sum_{\me}\sigma_0^{-1} (1-\alpha_E)^2|d_E|^2 \nonumber \\
 &&\times h_E^{1-d} \|\nabla u_h \cdot\bt_E\|_{0,E}^2,
\end{eqnarray} 
for some constant $\kappa_2$ which will be defined later.

Continuing similarly to Eq.~\eqref{eq:est_dh_01} and using in addition the inverse inequality Eq.~\eqref{eq:inv_est} leads to 
\begin{eqnarray}\label{eq:est_dh_03}
\lefteqn{\frac1\kappa_2\sum_{\me} \sigma_0 h_E^{d+1} \|\nabla (I_h u-u_h)\cdot\bt_E\|_{0,E}^2}\nonumber
\\
& \le & \frac{\sigma_0}{\kappa_2}\sum_{K \in \mathcal T_h} C_{\mathrm{edge}} C_{\mathrm{inv}}^2
\|I_h u-u_h\|_{0,K}^2\nonumber \\
& \le & \frac{2 C_{\mathrm{inv}}^2 C_{\mathrm{edge,max}} \left(1+(1+C_I)^2\right)}{\kappa_2} \|u-u_h\|_a^2.
\end{eqnarray}
Choosing 
\begin{equation}\label{eq:kappa_2}
\kappa_2 =  C_{\mathrm{inv}}^2 C_{\mathrm{edge,max}} \left(1+(1+C_I)^2\right)
\end{equation}
enables again to absorb this term multiplied with $(2C_Y)^{-1}$ in the left-hand side of Eq.~\eqref{eq:est_all_00}. Inserting Eq.~\eqref{eq:d_h_est} -- 
Eq.~\eqref{eq:kappa_2} in Eq.~\eqref{eq:est_all_00} one gets

\begin{eqnarray}\label{eq:est_all_temp}
\lefteqn{\|u-u_h\|_a^2}\nonumber\\
& \le &\frac{C^2_Y}{2(C_Y-2)}\sum_{K\in \mathcal{T}_h}\mathrm{min}\left\lbrace \frac{C_I^2}{\sigma_0},\ \frac{C_I^2h_K^2}{\varepsilon}\right\rbrace\|R_K(u_h)\|_{0,K}^2\nonumber\\
&& + \frac{C^2_Y}{2(C_Y-2)}\sum_{F\in \cF_h} \min \left\{\frac{C_F^2 h_F}{\varepsilon}, \frac{C_F^2}{\sigma_0^{1/2}\varepsilon^{1/2}} \right\}\| R_F(u_h)\|_{0,F}^2\\
&& + \frac{C_Y^2}{2(C_Y-2)}\sum_{\me}\min \Bigg\{\frac{\kappa_1h_E^2}{
\varepsilon}, \frac{\kappa_2}{\sigma_0}\Bigg\} (1-\alpha_E)^2|d_E|^2 h_E^{1-d} \|\nabla u_h \cdot\bt_E\|_{0,E}^2.\nonumber
\end{eqnarray}
Using standard calculus arguments one gets an optimal value of $C_Y=4$.

The estimates are summarized in the following theorem. 

\begin{theorem}[Global a posteriori error estimate] A global a posteriori error estimate 
for the energy norm is given by 
\begin{equation}\label{eq:global_upper_bound}
\|u-u_h\|_a^2\leq \eta_1^2+\eta_2^2+\eta_3^2,
\end{equation}
where
\begin{eqnarray*}
\eta_1^2& = & \sum_{K\in \mathcal{T}_h}\mathrm{min}\left\{ \frac{4 C_I^2}{\sigma_0},\ \frac{4C_I^2h_K^2}{\varepsilon}\right\}\|R_K(u_h)\|_{0,K}^2, \\ 
\eta_2^2 & =&   \sum_{F\in \cF_h} \min \left\{\frac{4 C_F^2 h_F}{\varepsilon}, \frac{4 C_F^2}{\sigma_0^{1/2}\varepsilon^{1/2}} \right\}\| R_F(u_h)\|_{0,F}^2,\\
\eta_3^2 & = & \sum_{\me}\min \Bigg\{\frac{4\kappa_1h_E^2}{
\varepsilon}, \frac{4\kappa_2}{\sigma_0}\Bigg\} (1-\alpha_E)^2|d_E|^2 h_E^{1-d} \|\nabla u_h \cdot\bt_E\|_{0,E}^2\nonumber,
\end{eqnarray*}
with $\kappa_1$ and $\kappa_2$ defined in Eq.~\eqref{eq:kappa_1} and Eq.~\eqref{eq:kappa_2}, respectively \revi{and $C_I,\ C_F$ are the non-computable constants arising from interpolation estimate and trace inequalities}.
\end{theorem}

\begin{proof}
The proof follows by inserting $C_Y=4$ in Eq.~\eqref{eq:est_all_temp}.
\end{proof}

\subsubsection{\revi{Formal Local Lower Bound}}\label{sec:lower_bound}
The posteriori estimator implied by the equation Eq.~\eqref{eq:global_upper_bound}
$$
\|u-u_h\|_a^2\leq C\sum_{K\in \mathcal{T}}\eta_K^2,
$$
provides a global upper bound on the discretization error up to the constant $C$. For using this estimator as the basis of an adaptive refinement algorithm, one wants the estimator to be efficient in the sense that $C$ is independent of the mesh size such that
$$
\eta_K^2\leq C\|u-u_h\|_{a,\omega_K}^2,
$$
where $\omega_K$ is some neighborhood of $K$. This type of bound is important as in conjunction with Eq.~\eqref{eq:global_upper_bound} it confirms that the rate of change of estimator as the mesh size is reduced matches the behavior of the actual error. If no such estimate is available, the estimator's performance is not optimal, and its use in the applications may result in poorly designed meshes. 

Consider a mesh cell $K$. Now the local estimator for mesh cell $K$ is defined as
\begin{equation}\label{eq:local_estimator}
\eta_K^2=\eta_{\mathrm{Int,K}}^2+\sum_{F\in \mathcal{F}_h(K)}\eta^2_{\mathrm{Face}, F}+\sum_{E\in \mathcal{E}_h(K)}\eta^2_{d_h, E}
\end{equation}
with
\begin{eqnarray}
\eta_{\mathrm{Int}, K}^2 & = & \mathrm{min}\left\lbrace \frac{4C_I^2}{\sigma_0},\frac{4C_I^2h_K^2}{\varepsilon}\right\rbrace \|R_{K,h}(u_h)\|_{0,K}^2,\nonumber \\
\eta_{\mathrm{Face}, F}^2 & = & \frac{1}{2}\mathrm{min}\left\lbrace \frac{4C_F^2h_F}{\varepsilon},\frac{4C_F^2}{\sigma^{1/2}_0\varepsilon^{1/2}}\right\rbrace \|R_F(u_h)\|_{0,F}^2,\nonumber \\
\eta_{d_h, E}^2 & = & \min \Bigg\{\frac{4\kappa_1h_E^2}{
\varepsilon}, \frac{4\kappa_2}{\sigma_0}\Bigg\} (1-\alpha_E)^2|d_E|^2 h_E^{1-d} \|\nabla u_h \cdot\bt_E\|_{0,E}^2,\nonumber \\
\end{eqnarray}
where $\mathcal{F}_h(K)$ is the set of all facets of $K$. Each inner facet belongs to two mesh cells, that's why the factor of $1/2$ is introduced.

The first two terms in the right of Eq.~\eqref{eq:local_estimator} are the standard interior and face residual terms that appear in a residual-based a posteriori error estimator for convection-diffusion equations. Using standard bubble function arguments introduced in \cite{Ver98}, one can bound these terms. For brevity we would not be deriving these bounds and only mention the final estimates.

For the interior residual one gets,
\begin{eqnarray}\label{eq:interior_residual_003}
\eta_{\mathrm{Int},K}
& \le & C\Bigg(\mathrm{max}\left\lbrace C_K^2+\frac{C_Kh_K}{\varepsilon}\|\bb\|_{\infty,K},\frac{C_K}{\sigma_0}\|c\|_{\infty,K}\right\rbrace\|u-u_h\|_{a,K}\\
&& +\frac{h_K}{\varepsilon^{1/2}}C_K\Big( \|f-f_h\|_{0,K}+\|(\bb-\bb_h)\cdot \nabla u_h\|_{0,K} +\|(c-c_h)u_h\|_{0,K}\Big)\Bigg),\nonumber
\end{eqnarray}
and for the face residual one gets,
\begin{eqnarray}\label{eq:face_residual_eq_003}
\eta_{\mathrm{Face},F}
& \le & C\Bigg(\mathrm{max}\left\lbrace C_{FB}+\frac{C_{FB}h_F\|\bb\|_{\infty,\omega_F}}{\varepsilon},\frac{C_{FB}h_F\|c\|_{\infty,\omega_F}}{\varepsilon^{1/2}\sigma_0^{1/2}}\right\rbrace \nonumber\\
&&\times \|u-u_h\|_{a,\omega_F} + \delta_{F\in \mathcal{F}_{h,N}}\frac{h_F^{1/2}}{\varepsilon^{1/2}}\|g-g_h\|_{0,F}\nonumber\\
&& + \sum_{K\in \omega_F}\Big[\eta_{\mathrm{Int},K}+\frac{h_K}{\varepsilon^{1/2}}\Big(\|f-f_h\|_{0,K} \nonumber \\
&& +\|(\bb-\bb_h)\cdot \nabla u_h\|_{0,K}+\|(c-c_h)u_h\|_{0,K}\Big)\Big]\Bigg),
\end{eqnarray}
where \revi{$\bb_h,c_h,f_h,$ and $g_h$ are approximations of the coefficients in the finite-dimensional space}, $C_K$ and $C_{\mathrm{FB}}$ are the constants appearing from bubble function arguments, \revi{$\delta_{F\in \mathcal{F}_{h,N}}$ is the Kronecker delta function which is one if the face belongs to the Neumann boundary,} and $C$ is a general constant independent of $h$.

\textbf{Edge Residuals:}\label{subsec:edge_estimates} The final term one wants to bound in $\eta_K$ is the AFC contribution. A similar term can be observed in \cite[Theorem~2]{ABR17}. Based on certain assumptions on the nonlinear stabilization, namely the Lipschitz continuity and linearity preservation, that term is bounded. We will not use such assumptions as they do not encompass the limiter that will be \revi{presented in the numerical simulations, namely the Kuzmin limiter}.

From the proof of \cite[Lemma~2]{BJKR18} we have
\begin{equation}\label{eq:d_E_bound}
|d_E|\leq C\left( \varepsilon +\|\bb\|_{\infty,\Omega}h+\|c\|_{\infty,\Omega}h^2\right) h_E^{d-2}.
\end{equation}
We have
\begin{eqnarray*}
\eta_{d_h, E} & \leq &C\sum_{E\in \mathcal{E}_h}(1-\alpha_E)|d_E|h_E^{(1-d)/2}\mathrm{min}\left\lbrace \frac{h_E}{\varepsilon^{1/2}},\frac{1}{\sigma_0^{1/2}}\right \rbrace\|\nabla u_h\cdot \bt_E\|_{0,E}.
\end{eqnarray*}

Hence, we get from Eq.~\eqref{eq:d_E_bound}
\begin{eqnarray}\label{eq:edge_residual}
\eta_{d_h, E} &\leq &C\sum_{E\in \mathcal{E}_h}(1-\alpha_E)\left( \varepsilon +\|\bb\|_{\infty,\Omega}h+\|c\|_{\infty,\Omega}h^2\right)\nonumber \\
&& \times \frac{h_E^{(3-d)/2}}{\varepsilon^{1/2}}\|\nabla u_h\cdot \bt_E\|_{0,E}\nonumber \\
& = & C\sum_{E\in \mathcal{E}_h}(1-\alpha_E)\left( \varepsilon^{1/2} +\frac{\|\bb\|_{\infty,\Omega}h}{\varepsilon^{1/2}}+\frac{\|c\|_{\infty,\Omega}h^2}{\varepsilon^{1/2}}\right)\nonumber \\
&& \times h_E^{(3-d)/2}\|\nabla u_h\cdot \bt_E\|_{0,E}.
\end{eqnarray}
For a fixed $\varepsilon$, we consider the convection-dominated regime, i.e., $\varepsilon\leq h$, then we get
$$
\eta_{d_h, E}=\mathcal{O}(h)
$$
in 2d, and 
$$
\eta_{d_h, E}=\mathcal{O}(h^{1/2})
$$
in 3d, whereas, for diffusion-dominated case we get $\mathcal{O}(h^{1/2})$ in 2d. This term is not exactly an oscillation. It is noted in \cite{BJK16} that the average rate of decay for the first factor in parentheses is one but no concrete analysis has been provided. Altogether this term has to be studied numerically. Also for shock-capturing methods a priori estimates usually give $\mathcal{O}(h^{1/2})$ convergence (see \cite[Corollary~17]{BJK16}), then we can expect the last term to behave as an oscillation (see \cite[Remark~5]{ABR17}). \revi{This is the reason we call this local lower bound a formal local lower bound}.

\begin{remark}
To simplify the notation we will denote $\eta_{d_h,E}$ by $\eta_{d_h}$ whenever we don't have ambiguity for $E$. Numerical examples will be presented in Sec.~\ref{sec:numres_post} to show the behavior of $\eta_{d_h}$.
\end{remark}

\begin{theorem}
There exists a constant $C>0$, independent of the size of elements of $\mathcal{T}$, such that, for every $K\in \mathcal{T}$, the following \revi{formal} local lower bound holds
\begin{eqnarray}\label{eq:lower_bound}
\lefteqn{\eta_{\mathrm{Int}, K}+\sum_{\revi{F}\in \mathcal{F}_h(K)}\eta_{\mathrm{Face},F}+\sum_{E\in \mathcal{E}_h(K)}\eta_{d_h,E}}\nonumber\\
& \le & \mathrm{max}\left\lbrace C_K^2+\frac{C_Kh_K}{\varepsilon}\|\bb\|_{\infty,K},\frac{C_K}{\sigma_0}\|c\|_{\infty,K}\right\rbrace \|u-u_h\|_{a,\omega_K}\nonumber\\
&& + C\sum_{\revi{K'}\in \omega_K}\frac{h_{K'}}{\varepsilon^{1/2}}\left(\|f-f_h\|_{0,K'}+\|(\bb-\bb_h)\cdot \nabla u_h\|_{0,K'}+\|(c-c_h)u_h\|_{0,K'}\right)\nonumber\\
&& + C\sum_{F\in \mathcal{F}_h(K)} \delta_{F\in \mathcal{F}_{h,N}}\frac{h_F^{1/2}}{\varepsilon^{1/2}}\|g-g_h\|_{0,F}\nonumber\\
&& + \sum_{E\in \mathcal{E}_{h}(K)}h^{1-d/2}\frac{h^{1/2}}{\varepsilon^{1/2}}\Big( \varepsilon+\|b\|_{\revi{\infty,\Omega}}h+\|c\|_{\revi{\infty,\Omega}}h^2\Big)\|\nabla u_h\cdot \bt_E\|_{0,E}.
\end{eqnarray}
\end{theorem}
\begin{proof}
This estimate can be obtained by combining Eq.~\eqref{eq:interior_residual_003}, Eq.~\eqref{eq:face_residual_eq_003}, and Eq.~\eqref{eq:edge_residual}.
\end{proof}

\begin{remark}
We note that the estimator is not robust with respect to $\varepsilon$. However, this is the usual case for a posteriori error estimators for the error measured in the energy norm. In \cite{TV15} residual-based a posteriori estimators for the error were proved robust with respect to a norm that includes a dual norm of the convective term. However, all the methods considered in \cite{TV15} were linear, and applying those techniques to nonlinear discretizations such as AFC does not seem feasible.
\end{remark}

\subsection{AFC-SUPG Estimator}\label{sec:afc_supg_est}
An alternative way of finding a global upper bound for the error in the energy norm for the AFC scheme is to use the estimator proposed in \cite{JN13}. An upper bound which is robust with respect to the diffusion coefficient, $\varepsilon$, was derived for the error in the SUPG norm \cite[Eq.~(11)]{JN13} for the SUPG scheme. It has been noted in \cite{JJ18} that choosing the initial solution as the SUPG solution for the nonlinear system of equations was the most appropriate. We exploit this fact to bound our error.

Let $u_{\mathrm{AFC}},\ u_{\mathrm{SUPG}}$ denote the AFC and SUPG solution, respectively. Then by the triangle inequality
\begin{equation*}
\begin{aligned}
\|u-u_{\mathrm{AFC}}\|_a^2 &\leq 2\left( \|u-u_{\mathrm{SUPG}}\|_a^2+\|u_{\mathrm{SUPG}}-u_{\mathrm{AFC}}\|_a^2\right)
\\&\leq 2\left( \|u-u_{\mathrm{SUPG}}\|_{\mathrm{SUPG}}^2+\|u_{\mathrm{SUPG}}-u_{\mathrm{AFC}}\|_a^2\right).
\end{aligned}
\end{equation*}
The first term can be bounded by \cite[Theorem~2.1]{JN13} and the second term is computable. Let
$$
\|u-u_{\mathrm{SUPG}}\|_{\mathrm{SUPG}}^2\leq \eta_{\mathrm{SUPG}}^2,
$$
where $\eta_{\mathrm{SUPG}}^2$ is given by \cite[Eq.~(36)]{JN13} and
$$
\eta_{\mathrm{AFC-SUPG}}:=\|u_{\mathrm{AFC}}-u_{\mathrm{SUPG}}\|_a,
$$
then
$$
\|u-u_{\mathrm{AFC}}\|_a^2\leq \eta^2,
$$
where
$$
\eta^2=2\left(\eta_{\mathrm{SUPG}}^2+\eta_{\mathrm{AFC-SUPG}}^2\right).
$$
Numerical simulations depicting the behavior of $\eta_{\mathrm{SUPG}}, \eta_{\mathrm{AFC-SUPG}}$ along with the adaptive refinement of grids will be presented in Sec.~\ref{sec:numres_post}. \revi{A local lower bound for this estimator will not be provided in this paper.}

\section{Numerical Studies}\label{sec:numres_post}
The standard strategy for numerically solving a partial differential equation on adaptively refined grids using an a posteriori error estimator is
$$
\mathbf{SOLVE}\rightarrow \mathbf{ESTIMATE}\rightarrow \mathbf{MARK}\rightarrow \mathbf{REFINE}.
$$
We note that to refine a grid adaptively, two important things are required:
\begin{itemize}
\item \emph{Marking strategy} that decides which mesh cells should be refined,
\item \emph{Refinement rules} which determines the actual subdivision of a mesh cell.
\end{itemize}
There are two marking strategies that are widely used in a posteriori packages, namely the \emph{maximum marking strategy} and the \emph{equilibration marking strategy} (see \cite{Ver13}). It is noted in \cite{Ver13} that both the strategies produce comparable results. Still, it is computationally cheaper to implement the maximum marking strategy, and hence it is used in our simulations. For refining of the mesh cells, \emph{red-green refinement} rules were used (see \cite{Ver13}) which would be referred to as conforming closure in the examples.

\begin{remark}
An issue that arises while marking cells for convection-dominated problems is that only a few mesh cells with a high error are marked, which deteriorates the algorithm's performance. To ensure that enough cells are marked, we follow the strategy prescribed in \cite[Sec.~4]{John00}.
\end{remark}

The quality of an estimator is usually judged by its global effectivity index that is given by,
$$
\eta_{\mathrm{eff}}=\frac{\eta}{\|u-u_h\|_a}.
$$
This index can be used to measure the quality of an estimator when the exact or a good approximation is known to the solution.

We note that we have the presence of certain constants in our estimators namely $C_I$ and $C_F$. We chose the value of these constants to be unity.

\begin{remark}
We have discussed two different strategies for finding a global upper bound for the AFC error in the energy norm. Further in this section we will refer to the residual-based estimator from Sec.~\ref{sec:upper_bound} as \emph{AFC-energy} technique and from Sec.~\ref{sec:afc_supg_est} as \emph{AFC-SUPG-energy} technique.
\end{remark}

\begin{remark}
One of the advantages of the nonlinear AFC schemes is that it produces a physically consistent solution. In the case of \cdr equations, it relates to the satisfaction of DMP. It has been noted in \cite{BJK16} that a sufficient condition for the satisfaction of DMP for the Kuzmin limiter is that the mesh is Delaunay in nature. With red-green refinements, subsequent refinement makes the mesh lose this property. One way around this is to use grids with hanging nodes. To the best of our knowledge, no theory or implementation has been suggested for continuous AFC schemes for steady-state \cdr equations in the context of hanging nodes.
\end{remark}

Numerical studies presented further in this section will comprehend the results for the two different techniques on the following conditions:
\begin{enumerate}
\item Compare the \afce and \afcse techniques:
\begin{enumerate}
\item with respect to the effectivity index in the energy norm.
\item with respect to adaptive grid refinement.
\end{enumerate}
\item Study the behavior of $\eta_{d_h}$ defined in Eq.~\eqref{eq:edge_residual}, on uniformly and adaptively refined grids.
\item Study the behavior of $\eta_{\mathrm{SUPG}}$ and $\eta_{\mathrm{AFC-SUPG}}$ for the \afcse technique.
\end{enumerate}

\begin{remark}
A comparative study for the solution of the nonlinear problem arising in the AFC schemes was performed in \cite{JJ18, JJ19}. It was found that the simplest fixed point iteration scheme was the most efficient. We present a brief overview of this scheme. The matrix formulation for Eq.~\eqref{eq:afc_def} is given by
\begin{equation}\label{eq:matrix_formulation}
Au+(I-\boldsymbol{\alpha})Du=F,
\end{equation}
where $A(=\lbrace a_{ij}\rbrace_{i,j=1}^N)$ is the stiffness matrix, $D(=\lbrace d_{ij}\rbrace_{i,j=1}^N)$ is the artificial diffusion matrix, $I$ is the identity matrix of size $N\times N$, $\boldsymbol{\alpha}(=\lbrace \alpha_{ij}\rbrace_{i,j=1}^N)$ is the limiter matrix, and $F$ is the right-hand side. Then one can re-write Eq.~\eqref{eq:matrix_formulation} and compute the next iterative solution as
$$
\left( A+D\right)u^{\nu+1}=F+\omega\boldsymbol{\alpha} Du^{\nu},
$$
where $\nu$ is the $\nu^{\mathrm{th}}$ iterative step and $\omega\in \mathbb{R}^+$ is a damping parameter. The matrix $A+D$ is a constant matrix, and hence for an iterative process, it can be factored once and can be used again in the iterative loop. A detailed overview can be found in \cite{JJ19} where it is referred to as fixed-point right-hand side. We will use this method for solving the nonlinear problems arising in our numerical simulations.
\end{remark}

The matrices were assembled exactly, and the linear systems were solved using the direct solver \textsc{UMFPACK}~\cite{Dav04}. The stopping criteria for the adaptive algorithm were either number of degrees of freedom $(\#\mathrm{dof})\gtrsim 10^6$ or $\eta<10^{-3}$. All the simulations were performed with the in-house code \textsc{ParMooN} \cite{WB16}.

\subsection{A Known 2d Solution with a Boundary Layer}\label{ex:known_2d_boundary}
This example was proposed in \cite[Example~1]{ABR17}. Consider $\varepsilon=10^{-3}$, $\bb=(2,1)^T$, $c=1$, $g=0$, $u_D=0$, and the right-hand side $f$ such that the exact solution is given by
$$
u(x,y)=y(1-y)\left( x-\frac{e^{(x-1)/\varepsilon}-e^{-1/\varepsilon}}{1-e^{-1/\varepsilon}}\right),
$$
on the domain $\Omega=(0,1)^2$ (see Fig.~\ref{fig:example_2d_known_sol}). An initial grid was defined with two triangles by joining the points $(0,0)$ and $(1,1)$. The simulations were started with a level 2 grid (i.e., $\# \mathrm{dof}=25$), initially uniform refinement was performed till level 4 (i.e., $\# \mathrm{dof}=289$). After that adaptive refinement was performed.

\begin{figure}[t!]
\centerline{\includegraphics[width=0.48\textwidth]{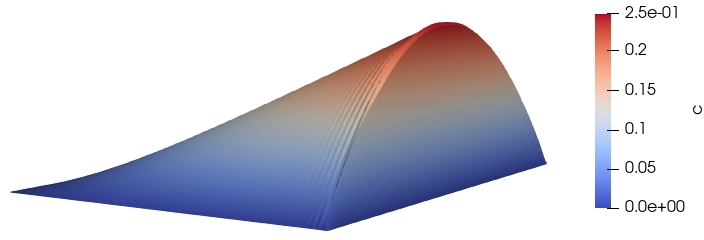}}
\caption{2d Boundary layer example. Solution (computed with the BJK limiter, level~7).}
\label{fig:example_2d_known_sol}
\end{figure}

First, we compare the behavior of effectivity indices for the \afce and \afcse techniques. For the \afce technique, we note that as the adaptive refinement starts, the effectivity index is high, and as the grid becomes refined the value decreases (see Fig.~\ref{fig:effectivity_index_2d_sol_boundary_layer} (left)). For the Kuzmin limiter on grids with fine adaptive regions $\eta_{\mathrm{eff}}\approx 232$ and for the BJK limiter $\eta_{\mathrm{eff}}\approx 12$. For the \afcse technique the values of the effectivity index are better than for the \afce technique (see Fig.~\ref{fig:effectivity_index_2d_sol_boundary_layer} (right)). One interesting observation to make is that the limiter does not play an important role in this technique. The values of effectivity indices are comparable for both the limiters. If the adaptive refinement is sufficiently fine, then for the Kuzmin limiter $\eta_{\mathrm{eff}}\approx 2$ and for the BJK limiter $\eta_{\mathrm{eff}}\approx 5$. 

\begin{figure}[t!]
\centerline{\begin{tikzpicture}[scale=0.75]
\begin{semilogxaxis}[
    legend pos=north east, xlabel = $\#\ \mathrm{dof}$, ylabel = $\eta_{\mathrm{eff}}$,
    legend cell align ={left}, title = {$\varepsilon=10^{-3}$}]
\addplot[color=red,  mark=oplus*, line width= 0.5mm,
dashdotted, mark options={scale=1.5,solid}] 
coordinates{(25.0, 157.798)( 81.0 , 152.499 )( 289.0 , 182.853 )( 344.0 , 1005.33 )( 435.0 , 866.463 )( 638.0 , 615.963 )( 1045.0 , 430.571 )( 1748.0 , 348.854 )( 2434.0 , 313.354 )( 3266.0 , 262.103 )( 3913.0 , 228.23 )( 5601.0 , 248.755 )( 6869.0 , 233.045 )( 10403.0 , 214.093 )( 13793.0 , 214.449 )( 19764.0 , 251.654 )( 22948.0 , 228.084 )( 27597.0 , 225.452 )( 34830.0 , 213.521 )( 43917.0 , 195.888 )( 54808.0 , 172.266 )( 67015.0 , 148.108 )( 80566.0 , 189.598 )( 104171.0 , 241.836 )( 124036.0 , 231.875 )( 143828.0 , 211.424 )( 177722.0 , 232.638 )( 210148.0 , 234.948 )( 247895.0 , 227.263 )( 294562.0 , 229.086 )( 353877.0 , 233.405 )( 422554.0 , 232.132 )( 505097.0 , 217.472 )( 623392.0 , 219.318 )( 791868.0 , 230.286 )( 974292.0 , 244.025 )( 1214238.0 , 232.149 )};
\addlegendentry{Kuzmin (Adaptive)} 
\addplot[color=blue,  mark=diamond*, line width= 0.5mm,
dashdotted, mark options={scale=1.5,solid}] 
coordinates{(25.0, 139.401)( 81.0 , 141.525 )( 289.0 , 170.581 )( 344.0 , 941.545 )( 423.0 , 824.347 )( 602.0 , 594.897 )( 950.0 , 413.008 )( 1236.0 , 359.739 )( 1723.0 , 267.727 )( 2664.0 , 191.521 )( 3147.0 , 143.894 )( 3728.0 , 69.0925 )( 5346.0 , 29.1357 )( 8415.0 , 21.2833 )( 10196.0 , 15.6269 )( 15483.0 , 12.4723 )( 22889.0 , 12.0138 )( 30350.0 , 11.997 )( 52411.0 , 11.9203 )( 70354.0 , 11.8922 )( 101062.0 , 11.9594 )( 154071.0 , 12.004 )( 213156.0 , 11.9671 )( 310335.0 , 11.8981 )( 391230.0 , 11.8986 )( 555210.0 , 11.9235 )( 759787.0 , 11.9572 )( 1011900.0 , 11.9926 )};
\addlegendentry{BJK (Adaptive)} 
\addplot[color=black,  mark=square*, line width= 0.5mm,
dashdotted, mark options={scale=1.5,solid}] 
coordinates{(25.0, 157.798)( 81.0 , 152.499 )( 289.0 , 182.853 )( 1089.0 , 258.84 )( 4225.0 , 219.16 )( 16641.0 , 150.702 )( 66049.0 , 103.2 )( 263169.0 , 79.311 )( 1050625.0 , 55.743 )};
\addlegendentry{Kuzmin (Uniform)}
\addplot[samples= 100000, color=dark_green,  mark=triangle*, line width= 0.5mm,
dashdotted, mark options={scale=1.5,solid}] 
coordinates{(25.0, 139.401)( 81.0 , 141.525 )( 289.0 , 170.581 )( 1089.0 , 244.127 )( 4225.0 , 210.508 )( 16641.0 , 144.71 )( 66049.0 , 93.8956 )( 263169.0 , 56.0329 )( 1050625.0 , 11.8521 )};
\addlegendentry{BJK (Uniform)} 
\end{semilogxaxis}
\end{tikzpicture}\hspace*{1em}
\begin{tikzpicture}[scale=0.75]
\begin{semilogxaxis}[
    legend pos=north east, xlabel = $\#\ \mathrm{dof}$, ylabel = $\eta_{\mathrm{eff}}$,
    legend cell align ={left},, title = {$\varepsilon=10^{-3}$}]
\addplot[color=salmon,  mark=oplus*, line width = 0.5mm, dashdotted,, mark options = {scale= 1.5, solid}] 
coordinates{(25.0, 174.631)( 81.0 , 112.566 )( 289.0 , 82.4627 )( 352.0 , 59.7597 )( 463.0 , 44.7054 )( 646.0 , 34.7865 )( 1045.0 , 24.7433 )( 1686.0 , 20.8614 )( 2022.0 , 19.1367 )( 3641.0 , 20.6808 )( 7491.0 , 24.0577 )( 14424.0 , 29.0774 )( 39409.0 , 18.1864 )( 49670.0 , 11.3822 )( 61741.0 , 7.64922 )( 83693.0 , 4.68844 )( 111831.0 , 4.36345 )( 150840.0 , 4.52452 )( 201913.0 , 3.00038 )( 269999.0 , 2.95759 )( 340794.0 , 3.12834 )( 429740.0 , 3.38819 )( 606132.0 , 2.26246 )( 845005.0 , 2.2052 )( 1114438.0 , 2.27671 )};
\addlegendentry{Kuzmin (Adaptive)} 
\addplot[color=royal_blue,  mark=diamond*, line width = 0.5mm, dashdotted,, mark options = {scale= 1.5, solid}] 
coordinates{(25.0, 176.195)( 81.0 , 112.288 )( 289.0 , 82.2871 )( 352.0 , 59.6964 )( 463.0 , 44.675 )( 646.0 , 34.7757 )( 1045.0 , 24.8238 )( 1686.0 , 20.9518 )( 2022.0 , 19.2404 )( 3641.0 , 20.4637 )( 7491.0 , 24.6255 )( 14363.0 , 29.3947 )( 39473.0 , 18.7249 )( 50843.0 , 11.3671 )( 63542.0 , 7.29906 )( 87090.0 , 5.71582 )( 117904.0 , 5.62455 )( 156698.0 , 5.25799 )( 210502.0 , 5.25344 )( 312177.0 , 5.26304 )( 462254.0 , 5.18561 )( 880881.0 , 5.15083 )( 1401761.0 , 5.17831 )};
\addlegendentry{BJK (Adaptive)} 
\addplot[color=lighter_gray,  mark=square*, line width = 0.5mm, dashdotted,, mark options = {scale= 1.5, solid}] 
coordinates{(25.0, 176.195)( 81.0 , 112.566 )( 289.0 , 82.4627 )( 1089.0 , 59.7289 )( 4225.0 , 44.5568 )( 16641.0 , 33.8111 )( 66049.0 , 23.2371 )( 263169.0 , 18.1528 )( 1050625.0 , 19.3542 )};
\addlegendentry{Kuzmin (Uniform)} 
\addplot[color=green,  mark=triangle*, line width = 0.5mm, dashdotted,, mark options = {scale= 1.5, solid}] 
coordinates{(25.0, 174.631)( 81.0 , 112.288 )( 289.0 , 82.2871 )( 1089.0 , 59.6387 )( 4225.0 , 44.5321 )( 16641.0 , 33.8249 )( 66049.0 , 23.3362 )( 263169.0 , 18.2639 )( 1050625.0 , 18.9198 )};
\addlegendentry{BJK (Uniform)} 
\end{semilogxaxis}
\end{tikzpicture}}
\caption{Example~\ref{ex:known_2d_boundary}: Effectivity index in the energy norm with \afce technique defined in Sec.~\ref{sec:upper_bound} (left) and \afcse technique defined in Sec.~\ref{sec:afc_supg_est} (right).}\label{fig:effectivity_index_2d_sol_boundary_layer}
\end{figure}
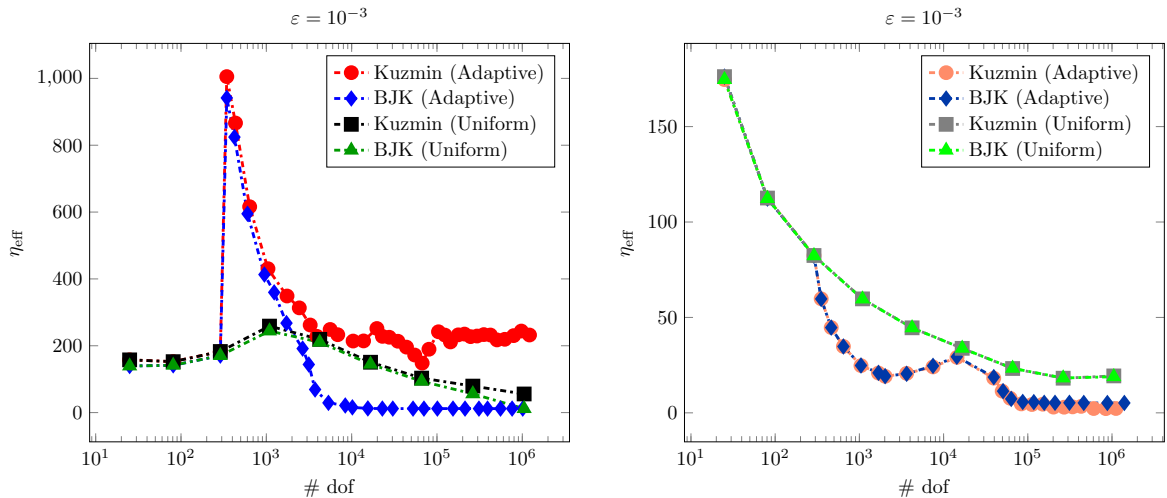

Next, we look at the individual behavior of $\eta_{\mathrm{SUPG}}$ and $\eta_{\mathrm{AFC-SUPG}}$. It can be seen in Fig.~\ref{fig:example_1_ABR17_supg_eta_afc_eta_supg_comp} that the dominating term is $\eta_{\mathrm{SUPG}}$ and hence, the AFC contribution, $\eta_{\mathrm{AFC-SUPG}}$ does not play a pivotal role in the effectivity index and the refinement of the grid.

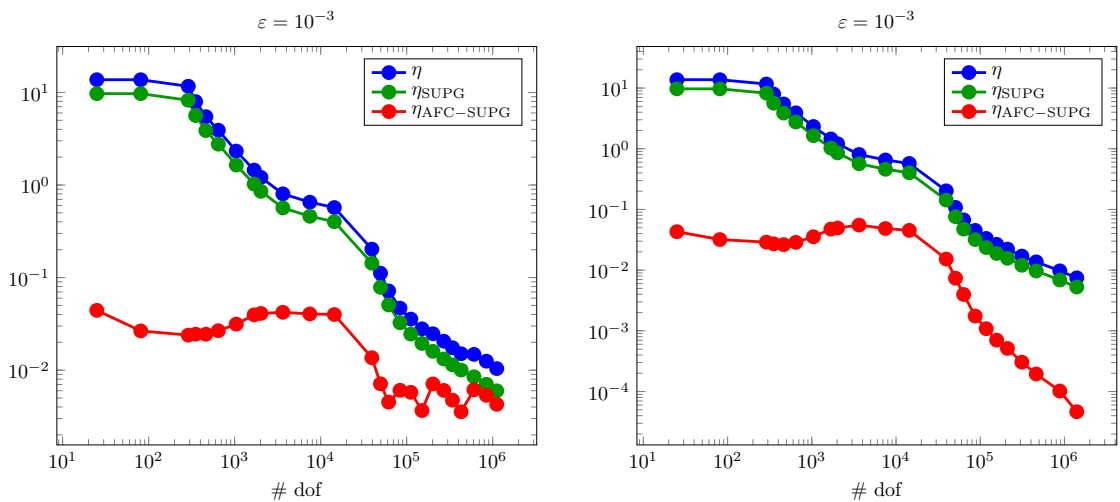
\begin{figure}[t!]
\centerline{\begin{tikzpicture}[scale=0.75]
\begin{loglogaxis}[
    legend pos=north east, xlabel = $\#\ \mathrm{dof}$,
    legend cell align ={left},, title = {$\varepsilon=10^{-3}$}]
\addplot[color=blue,  mark=oplus*, line width = 0.5mm, mark options = {scale= 1.5, solid}] 
coordinates{(25.0, 13.7241)( 81.0 , 13.7158 )( 289.0 , 11.6547 )( 352.0 , 7.96451 )( 463.0 , 5.47235 )( 646.0 , 3.90816 )( 1045.0 , 2.32506 )( 1686.0 , 1.44996 )( 2022.0 , 1.20874 )( 3641.0 , 0.802179 )( 7491.0 , 0.652303 )( 14424.0 , 0.57113 )( 39409.0 , 0.203022 )( 49670.0 , 0.111413 )( 61741.0 , 0.0717912 )( 83693.0 , 0.0467526 )( 111831.0 , 0.0356211 )( 150840.0 , 0.0279488 )( 201913.0 , 0.0247254 )( 269999.0 , 0.020554 )( 340794.0 , 0.0174715 )( 429740.0 , 0.0150476 )( 606132.0 , 0.0148084 )( 845005.0 , 0.0125027 )( 1114438.0 , 0.0104006 )};
\addlegendentry{$\eta$} 
\addplot[color=dark_green,  mark=oplus*, line width = 0.5mm, mark options = {scale= 1.5, solid}] 
coordinates{(25.0, 9.70433)( 81.0 , 9.69851 )( 289.0 , 8.24112 )( 352.0 , 5.6317 )( 463.0 , 3.86946 )( 646.0 , 2.76336 )( 1045.0 , 1.64376 )( 1686.0 , 1.02451 )( 2022.0 , 0.853725 )( 3641.0 , 0.565656 )( 7491.0 , 0.459467 )( 14424.0 , 0.401873 )( 39409.0 , 0.142911 )( 49670.0 , 0.078458 )( 61741.0 , 0.0505631 )( 83693.0 , 0.0324988 )( 111831.0 , 0.0245197 )( 150840.0 , 0.01942 )( 201913.0 , 0.0159777 )( 269999.0 , 0.013216 )( 340794.0 , 0.0114104 )( 429740.0 , 0.0100282 )( 606132.0 , 0.00848166 )( 845005.0 , 0.00704 )( 1114438.0 , 0.00598289 )};
\addlegendentry{$\eta_{\mathrm{SUPG}}$} 
\addplot[color=red,  mark=oplus*, line width = 0.5mm, mark options = {scale= 1.5, solid}] 
coordinates{(25.0, 0.0442584)( 81.0 , 0.0265415 )( 289.0 , 0.023912 )( 352.0 , 0.0244473 )( 463.0 , 0.0244581 )( 646.0 , 0.0266655 )( 1045.0 , 0.0314251 )( 1686.0 , 0.0395718 )( 2022.0 , 0.0409992 )( 3641.0 , 0.0421705 )( 7491.0 , 0.0404924 )( 14424.0 , 0.0399103 )( 39409.0 , 0.0136131 )( 49670.0 , 0.00712439 )( 61741.0 , 0.00451182 )( 83693.0 , 0.00606038 )( 111831.0 , 0.0057632 )( 150840.0 , 0.00366488 )( 201913.0 , 0.00709832 )( 269999.0 , 0.00604747 )( 340794.0 , 0.00473603 )( 429740.0 , 0.0035567 )( 606132.0 , 0.00614051 )( 845005.0 , 0.00534762 )( 1114438.0 , 0.00427678 )};
\addlegendentry{$\eta_{\mathrm{AFC-SUPG}}$} 
\end{loglogaxis}
\end{tikzpicture}\hspace*{1em}
\begin{tikzpicture}[scale=0.75]
\begin{loglogaxis}[
    legend pos=north east, xlabel = $\#\ \mathrm{dof}$,
    legend cell align ={left},, title = {$\varepsilon=10^{-3}$}]
\addplot[color=blue,  mark=oplus*, line width = 0.5mm, mark options = {scale= 1.5, solid}] 
coordinates{(25.0, 13.7241)( 81.0 , 13.7158 )( 289.0 , 11.6548 )( 352.0 , 7.96452 )( 463.0 , 5.47236 )( 646.0 , 3.90819 )( 1045.0 , 2.32517 )( 1686.0 , 1.45043 )( 2022.0 , 1.20937 )( 3641.0 , 0.80375 )( 7491.0 , 0.653383 )( 14363.0 , 0.57218 )( 39473.0 , 0.202575 )( 50843.0 , 0.108042 )( 63542.0 , 0.0672869 )( 87090.0 , 0.0451663 )( 117904.0 , 0.0333104 )( 156698.0 , 0.0266201 )( 210502.0 , 0.02221 )( 312177.0 , 0.0170385 )( 462254.0 , 0.0136912 )( 880881.0 , 0.00972972 )( 1401761.0 , 0.0074619 )};
\addlegendentry{$\eta$} 
\addplot[color=dark_green,  mark=oplus*, line width = 0.5mm, mark options = {scale= 1.5, solid}] 
coordinates{(25.0, 9.70453)( 81.0 , 9.69851 )( 289.0 , 8.24112 )( 352.0 , 5.6317 )( 463.0 , 3.86946 )( 646.0 , 2.76336 )( 1045.0 , 1.64376 )( 1686.0 , 1.02451 )( 2022.0 , 0.853723 )( 3641.0 , 0.565654 )( 7491.0 , 0.459463 )( 14363.0 , 0.402057 )( 39473.0 , 0.14243 )( 50843.0 , 0.076041 )( 63542.0 , 0.0474127 )( 87090.0 , 0.0318893 )( 117904.0 , 0.023529 )( 156698.0 , 0.0188101 )( 210502.0 , 0.0156965 )( 312177.0 , 0.0120442 )( 462254.0 , 0.00967919 )( 880881.0 , 0.0068792 )( 1401761.0 , 0.00527616 )};
\addlegendentry{$\eta_{\mathrm{SUPG}}$} 
\addplot[color=red,  mark=oplus*, line width = 0.5mm, mark options = {scale= 1.5, solid}] 
coordinates{(25.0, 0.0429842)( 81.0 , 0.0319308 )( 289.0 , 0.0288995 )( 352.0 , 0.0271393 )( 463.0 , 0.0262942 )( 646.0 , 0.0286865 )( 1045.0 , 0.0354736 )( 1686.0 , 0.0474512 )( 2022.0 , 0.0494385 )( 3641.0 , 0.0551615 )( 7491.0 , 0.0484658 )( 14363.0 , 0.0452204 )( 39473.0 , 0.0152351 )( 50843.0 , 0.00736799 )( 63542.0 , 0.00397464 )( 87090.0 , 0.00175224 )( 117904.0 , 0.0010843 )( 156698.0 , 0.000702938 )( 210502.0 , 0.00051347 )( 312177.0 , 0.000304677 )( 462254.0 , 0.000194678 )( 880881.0 , 0.000101997 )( 1401761.0 , 4.6501e-05 )};
\addlegendentry{$\eta_{\mathrm{AFC-SUPG}}$} 
\end{loglogaxis}
\end{tikzpicture}}
\caption{Example~\ref{ex:known_2d_boundary}: Comparison of $\eta_{\mathrm{SUPG}}$ and $\eta_{\mathrm{AFC-SUPG}}$ for \afcse technique. Kuzmin limiter (left) and BJK limiter (right).}\label{fig:example_1_ABR17_supg_eta_afc_eta_supg_comp}
\end{figure}

Then, we study the behavior of the error in the energy norm, its relation to the a posteriori error estimates, and the behavior of the part $\eta_{d_h}$ of the error estimators in some detail. One can observe that $\|u-u_h\|_a$, $\eta_{d_h}$, and $\eta$ for the \afce technique decay optimally on adaptive grids for the BJK limiter (see Fig.~\ref{fig:error_energy_norm_example_2d_sol_boundary_layer_afc} (left)). For the Kuzmin limiter one observes that as the grid becomes fine, the optimal rate is not obtained for $\|u-u_h\|_a$, $\eta_{d_h}$ and $\eta$. It has been noted in \cite[Remark~18]{BJK16} that if the grid is non-Delaunay and the problem becomes diffusion-dominated then the AFC method with the Kuzmin limiter fails to converge. With successive refinement of the grid, the problem becomes locally diffusion-dominated (in the sense of a small grid Peclet number) and one has to expect, because of the conforming closure and the resulting obtuse angles, that there is no convergence. The error estimator with the \afce technique predicts this irregular behavior of the error. This reduction of the rate of convergence is not observed while using BJK limiter. 

We also note that for the Kuzmin limiter, $\eta_{d_h}$ is comparable with $\eta$ and hence is the leading term in the adaptive refinement of the grid. For the BJK limiter, as the grid becomes finer, $\eta_{d_h}$ is small as compared to $\eta$.

\revi{After studying the behavior of the errors we comment on the behavior of the effectivity index presented in Fig.~\ref{fig:effectivity_index_2d_sol_boundary_layer}. We note that the effectivity index for the adaptive approach is better for
the BJK limiter from around $3000$ degrees of freedom. It is only worse
for coarse grids. Fig.~\ref{fig:error_energy_norm_example_2d_sol_boundary_layer_afc} (right) shows the errors on both the uniform and adaptive grids. In Fig.~\ref{fig:error_energy_norm_example_2d_sol_boundary_layer_afc} (right) it’s clear that the errors on the adaptive grids are smaller. For the Kuzmin limiter, the effectivity index on the adaptive grid is always larger than on the uniform grid. It reflects very well that the method does not converge. Comparing Fig.~\ref{fig:effectivity_index_2d_sol_boundary_layer} (left) and Fig.~\ref{fig:error_energy_norm_example_2d_sol_boundary_layer_afc} (left) one can note that the effectivity index on the adaptive grid is always larger if $\eta_{d_h}$ dominates the error estimates. Thus, one can guess that $\eta_{d_h}$ might lead to a stronger overestimate of the error on adaptive grids than on uniform grids, i.e., $\kappa_1$ and $\kappa_2$ (see Eq.~\eqref{eq:global_upper_bound}) might be more accurate approximations on uniform grids where all mesh cells are identical.}

\begin{figure}[t!]
\centerline{\begin{tikzpicture}[scale=0.75]
\begin{loglogaxis}[
    legend pos=south west, xlabel = $\#\ \mathrm{dof}$,
    legend cell align ={left},, title = {$\varepsilon=10^{-3}$},
    legend style={nodes={scale=0.75, transform shape}}]
\addplot[color=red,  mark=square*, line width = 0.5mm, dashdotted,,mark options = {scale= 1.0, solid}]
coordinates{(25.0, 0.07858936443668)( 81.0 , 0.122148359896 )( 289.0 , 0.141635427858 )( 344.0 , 0.133503427326 )( 423.0 , 0.123737496444 )( 602.0 , 0.113186520987 )( 950.0 , 0.0955276434282 )( 1236.0 , 0.0850302895705 )( 1723.0 , 0.0673164611738 )( 2664.0 , 0.0506826934128 )( 3147.0 , 0.043885675041 )( 3728.0 , 0.0382473606946 )( 5346.0 , 0.0368609571148 )( 8415.0 , 0.0259327038145 )( 10196.0 , 0.0221249203781 )( 15483.0 , 0.0169535474508 )( 22889.0 , 0.0127992815453 )( 30350.0 , 0.0107210025579 )( 52411.0 , 0.00812162742894 )( 70354.0 , 0.00679989946714 )( 101062.0 , 0.0055036592655 )( 154071.0 , 0.00451500782148 )( 213156.0 , 0.00378968793419 )( 310335.0 , 0.00306560521585 )( 391230.0 , 0.00272378607707 )( 555210.0 , 0.00231832831478 )( 759787.0 , 0.00196249536785 )( 1011900.0 , 0.00166935796313 )};
\addlegendentry{$\|u-u_h\|_a$ (BJK)} 
\addplot[color=blue,  mark=square*, line width = 0.5mm, dashdotted,,mark options = {scale= 1.0, solid}] 
coordinates{(25.0, 10.9555)( 81.0 , 17.287 )( 289.0 , 24.1603 )( 344.0 , 125.7 )( 423.0 , 102.003 )( 602.0 , 67.3343 )( 950.0 , 39.4537 )( 1236.0 , 30.5887 )( 1723.0 , 18.0224 )( 2664.0 , 9.70681 )( 3147.0 , 6.31488 )( 3728.0 , 2.64261 )( 5346.0 , 1.07397 )( 8415.0 , 0.551935 )( 10196.0 , 0.345744 )( 15483.0 , 0.21145 )( 22889.0 , 0.153768 )( 30350.0 , 0.12862 )( 52411.0 , 0.0968122 )( 70354.0 , 0.0808655 )( 101062.0 , 0.0658205 )( 154071.0 , 0.054198 )( 213156.0 , 0.0453515 )( 310335.0 , 0.036475 )( 391230.0 , 0.0324092 )( 555210.0 , 0.0276426 )( 759787.0 , 0.0234659 )( 1011900.0 , 0.02002 )};
\addlegendentry{$\eta$ (BJK)} 
\addplot[color=dark_green,  mark=square*, line width = 0.5mm, dashdotted,,mark options = {scale= 1.0, solid}] 
coordinates{(25.0, 10.2718)( 81.0 , 15.583 )( 289.0 , 22.4177 )( 344.0 , 125.398 )( 423.0 , 101.815 )( 602.0 , 67.2451 )( 950.0 , 39.4032 )( 1236.0 , 30.5509 )( 1723.0 , 17.9921 )( 2664.0 , 9.68029 )( 3147.0 , 6.28703 )( 3728.0 , 2.59964 )( 5346.0 , 0.976864 )( 8415.0 , 0.455636 )( 10196.0 , 0.222329 )( 15483.0 , 0.059851 )( 22889.0 , 0.0211256 )( 30350.0 , 0.0126068 )( 52411.0 , 0.00887979 )( 70354.0 , 0.00649381 )( 101062.0 , 0.00570253 )( 154071.0 , 0.00354036 )( 213156.0 , 0.00169654 )( 310335.0 , 0.00090679 )( 391230.0 , 0.000544933 )( 555210.0 , 0.00046401 )( 759787.0 , 0.000413572 )( 1011900.0 , 0.000251337 )};
\addlegendentry{$\eta_{d_h}$ (BJK)} 
\addplot[color=gold,  mark=oplus*, line width = 0.5mm, dashdotted,,mark options = {scale= 1.0, solid}]
coordinates{(25.0, 0.077891666488992)( 81.0 , 0.121846769481 )( 289.0 , 0.14133350854 )( 344.0 , 0.133315385795 )( 435.0 , 0.12273318655 )( 638.0 , 0.11229010666 )( 1045.0 , 0.0937867454382 )( 1748.0 , 0.0671275985206 )( 2434.0 , 0.0549707763066 )( 3266.0 , 0.041619680414 )( 3913.0 , 0.0370608336357 )( 5601.0 , 0.030861820373 )( 6869.0 , 0.0259637827384 )( 10403.0 , 0.0212224875198 )( 13793.0 , 0.0194389999956 )( 19764.0 , 0.0168472600542 )( 22948.0 , 0.0145585699812 )( 27597.0 , 0.0131627082309 )( 34830.0 , 0.0117222059116 )( 43917.0 , 0.0106538064624 )( 54808.0 , 0.00988919183729 )( 67015.0 , 0.00951595969239 )( 80566.0 , 0.0106263144336 )( 104171.0 , 0.0117884042022 )( 124036.0 , 0.0124765147484 )( 143828.0 , 0.00892467901198 )( 177722.0 , 0.00854716998355 )( 210148.0 , 0.00876603924271 )( 247895.0 , 0.00880931390225 )( 294562.0 , 0.00870837832573 )( 353877.0 , 0.00873623402415 )( 422554.0 , 0.00852636325256 )( 505097.0 , 0.00471629905189 )( 623392.0 , 0.00429238759794 )( 791868.0 , 0.00486206804552 )( 974292.0 , 0.0049507468671 )( 1214238.0 , 0.00516956170283 )};
\addlegendentry{$\|u-u_h\|_a$ (Kuzmin)} 
\addplot[color=navy_blue,  mark=oplus*, line width = 0.5mm, dashdotted,,mark options = {scale= 1.0, solid}] 
coordinates{(25.0, 12.2911)( 81.0 , 18.5816 )( 289.0 , 25.8432 )( 344.0 , 134.026 )( 435.0 , 106.344 )( 638.0 , 69.1666 )( 1045.0 , 40.3819 )( 1748.0 , 23.4178 )( 2434.0 , 17.2253 )( 3266.0 , 10.9086 )( 3913.0 , 8.45839 )( 5601.0 , 7.67703 )( 6869.0 , 6.05072 )( 10403.0 , 4.54359 )( 13793.0 , 4.16867 )( 19764.0 , 4.23968 )( 22948.0 , 3.32058 )( 27597.0 , 2.96756 )( 34830.0 , 2.50294 )( 43917.0 , 2.08695 )( 54808.0 , 1.70357 )( 67015.0 , 1.40939 )( 80566.0 , 2.01473 )( 104171.0 , 2.85086 )( 124036.0 , 2.893 )( 143828.0 , 1.88689 )( 177722.0 , 1.98839 )( 210148.0 , 2.05957 )( 247895.0 , 2.00203 )( 294562.0 , 1.99497 )( 353877.0 , 2.03908 )( 422554.0 , 1.97924 )( 505097.0 , 1.02566 )( 623392.0 , 0.941397 )( 791868.0 , 1.11967 )( 974292.0 , 1.20811 )( 1214238.0 , 1.20011 )};
\addlegendentry{$\eta$ (Kuzmin)} 
\addplot[color=purple,  mark=oplus*, line width = 0.5mm,dashdotted,, mark options = {scale= 1.0, solid}]
coordinates{(25.0, 11.6814)( 81.0 , 17.0304 )( 289.0 , 24.258 )( 344.0 , 133.749 )( 435.0 , 106.179 )( 638.0 , 69.0992 )( 1045.0 , 40.3521 )( 1748.0 , 23.4012 )( 2434.0 , 17.2135 )( 3266.0 , 10.899 )( 3913.0 , 8.44915 )( 5601.0 , 7.67003 )( 6869.0 , 6.04378 )( 10403.0 , 4.53711 )( 13793.0 , 4.16269 )( 19764.0 , 4.23531 )( 22948.0 , 3.31625 )( 27597.0 , 2.96348 )( 34830.0 , 2.49918 )( 43917.0 , 2.08325 )( 54808.0 , 1.69956 )( 67015.0 , 1.40502 )( 80566.0 , 2.01202 )( 104171.0 , 2.84929 )( 124036.0 , 2.89171 )( 143828.0 , 1.88529 )( 177722.0 , 1.98716 )( 210148.0 , 2.05849 )( 247895.0 , 2.00104 )( 294562.0 , 1.99413 )( 353877.0 , 2.03839 )( 422554.0 , 1.97872 )( 505097.0 , 1.02486 )( 623392.0 , 0.940574 )( 791868.0 , 1.11903 )( 974292.0 , 1.20756 )( 1214238.0 , 1.1996 )};
\addlegendentry{$\eta_{d_h}$ (Kuzmin)} 
\addplot[color=violet,  line width = 0.5mm, dashdotted,,mark options = {scale= 1.0, solid}]
coordinates{(25.0, 0.2)( 81.0 , 0.111111111111 )( 289.0 , 0.0588235294118 )( 344.0 , 0.0539163866017 )( 423.0 , 0.0486216638326 )( 602.0 , 0.040756957297 )( 950.0 , 0.0324442842262 )( 1236.0 , 0.0284440061994 )( 1723.0 , 0.0240911405462 )( 2664.0 , 0.0193746064573 )( 3147.0 , 0.0178259066765 )( 3728.0 , 0.016378044552 )( 5346.0 , 0.013676832331 )( 8415.0 , 0.0109011656695 )( 10196.0 , 0.00990341746674 )( 15483.0 , 0.00803660166676 )( 22889.0 , 0.00660977369547 )( 30350.0 , 0.0057401157793 )( 52411.0 , 0.00436806182759 )( 70354.0 , 0.00377012372505 )( 101062.0 , 0.00314561853888 )( 154071.0 , 0.0025476487422 )( 213156.0 , 0.00216596392537 )( 310335.0 , 0.00179508335809 )( 391230.0 , 0.00159876239741 )( 555210.0 , 0.00134205823153 )( 759787.0 , 0.00114723944509 )( 1011900.0 , 0.000994102582563 )};
\addlegendentry{Optimal rate $\mathcal{O}(h)$} 
\end{loglogaxis}
\end{tikzpicture}\hspace*{1em}
\begin{tikzpicture}[scale=0.75]
\begin{loglogaxis}[
    legend pos=south west, xlabel = $\#\ \mathrm{dof}$,
    legend cell align ={left},, title = {$\varepsilon=10^{-3}$},
    legend style={nodes={scale=0.75, transform shape}}]
\addplot[color=red,  mark=square*, line width = 0.5mm, dashdotted,,mark options = {scale= 1.5, solid}]
coordinates{(25.0, 0.07858936443668)( 81.0 , 0.122148359896 )( 289.0 , 0.141635427858 )( 1089.0 , 0.133469285398 )( 4225.0 , 0.122404454653 )( 16641.0 , 0.111686213105 )( 66049.0 , 0.0921912406045 )( 263169.0 , 0.062084993431 )( 1050625.0 , 0.0364598809841 )};
\addlegendentry{$\|u-u_h\|_a$ (BJK+Uniform)} 
\addplot[color=blue,  mark=square*, line width = 0.5mm, dashdotted,,mark options = {scale= 1.5, solid}] 
coordinates{(25.0, 0.07858936443668)( 81.0 , 0.122148359896 )( 289.0 , 0.141635427858 )( 344.0 , 0.133503427326 )( 423.0 , 0.123737496444 )( 602.0 , 0.113186520987 )( 950.0 , 0.0955276434282 )( 1236.0 , 0.0850302895705 )( 1723.0 , 0.0673164611738 )( 2664.0 , 0.0506826934128 )( 3147.0 , 0.043885675041 )( 3728.0 , 0.0382473606946 )( 5346.0 , 0.0368609571148 )( 8415.0 , 0.0259327038145 )( 10196.0 , 0.0221249203781 )( 15483.0 , 0.0169535474508 )( 22889.0 , 0.0127992815453 )( 30350.0 , 0.0107210025579 )( 52411.0 , 0.00812162742894 )( 70354.0 , 0.00679989946714 )( 101062.0 , 0.0055036592655 )( 154071.0 , 0.00451500782148 )( 213156.0 , 0.00378968793419 )( 310335.0 , 0.00306560521585 )( 391230.0 , 0.00272378607707 )( 555210.0 , 0.00231832831478 )( 759787.0 , 0.00196249536785 )( 1011900.0 , 0.00166935796313 )};
\addlegendentry{$\|u-u_h\|_a$ (BJK+Adaptive)} 
\addplot[color=gold,  mark=oplus*, line width = 0.5mm, dashdotted,,mark options = {scale= 1.5, solid}]
coordinates{(25.0, 0.077891666488992)( 81.0 , 0.121846769481 )( 289.0 , 0.14133350854 )( 1089.0 , 0.133267362871 )( 4225.0 , 0.122336162121 )( 16641.0 , 0.11173100268 )( 66049.0 , 0.0925789492918 )( 263169.0 , 0.0624279199852 )( 1050625.0 , 0.0355373089066 )};
\addlegendentry{$\|u-u_h\|_a$ (Kuzmin+Uniform)} 
\addplot[color=navy_blue,  mark=oplus*, line width = 0.5mm, dashdotted,,mark options = {scale= 1.5, solid}] 
coordinates{(25.0, 0.077891666488992)( 81.0 , 0.121846769481 )( 289.0 , 0.14133350854 )( 344.0 , 0.133315385795 )( 435.0 , 0.12273318655 )( 638.0 , 0.11229010666 )( 1045.0 , 0.0937867454382 )( 1748.0 , 0.0671275985206 )( 2434.0 , 0.0549707763066 )( 3266.0 , 0.041619680414 )( 3913.0 , 0.0370608336357 )( 5601.0 , 0.030861820373 )( 6869.0 , 0.0259637827384 )( 10403.0 , 0.0212224875198 )( 13793.0 , 0.0194389999956 )( 19764.0 , 0.0168472600542 )( 22948.0 , 0.0145585699812 )( 27597.0 , 0.0131627082309 )( 34830.0 , 0.0117222059116 )( 43917.0 , 0.0106538064624 )( 54808.0 , 0.00988919183729 )( 67015.0 , 0.00951595969239 )( 80566.0 , 0.0106263144336 )( 104171.0 , 0.0117884042022 )( 124036.0 , 0.0124765147484 )( 143828.0 , 0.00892467901198 )( 177722.0 , 0.00854716998355 )( 210148.0 , 0.00876603924271 )( 247895.0 , 0.00880931390225 )( 294562.0 , 0.00870837832573 )( 353877.0 , 0.00873623402415 )( 422554.0 , 0.00852636325256 )( 505097.0 , 0.00471629905189 )( 623392.0 , 0.00429238759794 )( 791868.0 , 0.00486206804552 )( 974292.0 , 0.0049507468671 )( 1214238.0 , 0.00516956170283 )};
\addlegendentry{$\|u-u_h\|_a$ (Kuzmin+Adaptive)} 
\addplot[color=violet,  line width = 0.5mm, dashdotted,,mark options = {scale= 1.5, solid}]
coordinates{(25.0, 0.2)( 81.0 , 0.111111111111 )( 289.0 , 0.0588235294118 )( 1089.0 , 0.030303030303 )( 4225.0 , 0.0153846153846 )( 16641.0 , 0.0077519379845 )( 66049.0 , 0.00389105058366 )( 263169.0 , 0.00194931773879 )( 1050625.0 , 0.000975609756098 )};
\addlegendentry{Optimal rate $\mathcal{O}(h)$} 
\end{loglogaxis}
\end{tikzpicture}}
\caption{Example~\ref{ex:known_2d_boundary}: Error in energy norm with \afce technique defined in Sec.~\ref{sec:upper_bound}. The line corresponding to $\eta$ (Kuzmin) is below $\eta_{d_h}$ (Kuzmin) in the left figure.}\label{fig:error_energy_norm_example_2d_sol_boundary_layer_afc}
\end{figure}
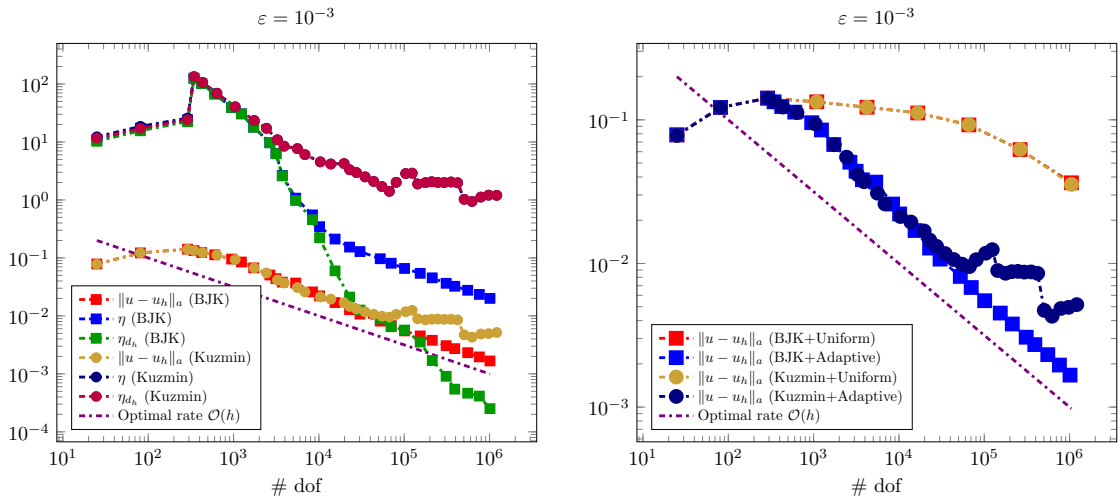

For the \afcse technique the error and $\eta$ values are shown in Fig.~\ref{fig:error_comparison_example_2d_sol_boundary_layer_afcse} (left). For the Kuzmin limiter, similar observation to the \afce technique can be made. One issue to note is that the estimator($\eta$) with \afcse technique does not predict the irregular behavior. It has already been mentioned that the AFC contribution does not play an important role here. \revi{Also, $\eta_{d_h}$ is absent from the \afcse technique. Hence, the effectivity
index is smaller on adaptive grids as compared to the uniform grids.}

Fig.~\ref{fig:adaptive_grids_example_1_ABR17} shows the $14^{\mathrm{th}}$ adaptively refined grid with \afce technique. One can observe obtuse angles in the adaptive grids.

\begin{figure}[t!]
\centerline{\begin{tikzpicture}[scale=0.75]
\begin{loglogaxis}[
    legend pos=south west, xlabel = $\#\ \mathrm{dof}$,
    legend cell align ={left},, title = {$\varepsilon=10^{-3}$},
    legend style={nodes={scale=0.75, transform shape}}]
\addplot[color=red,  mark=square*, line width = 0.5mm, dashdotted,,mark options = {scale= 1.0, solid}]
coordinates{(25.0, 0.07858936443668)( 81.0 , 0.122148359896 )( 289.0 , 0.141635427858 )( 352.0 , 0.133417158202 )( 463.0 , 0.122492774489 )( 646.0 , 0.112382684179 )( 1045.0 , 0.0936672027211 )( 1686.0 , 0.0692271478092 )( 2022.0 , 0.0628558741378 )( 3641.0 , 0.0392768611604 )( 7491.0 , 0.0265327634077 )( 14363.0 , 0.0194653971443 )( 39473.0 , 0.0108185060683 )( 50843.0 , 0.0095048163167 )( 63542.0 , 0.00921857759092 )( 87090.0 , 0.00790197230352 )( 117904.0 , 0.00592231489241 )( 156698.0 , 0.00506278044293 )( 210502.0 , 0.0042277082571 )( 312177.0 , 0.00323738248665 )( 462254.0 , 0.00264023149239 )( 880881.0 , 0.00188896077616 )( 1401761.0 , 0.0014409911019 )};
\addlegendentry{$\|u-u_h\|_a$ (BJK)} 
\addplot[color=blue,  mark=square*, line width = 0.5mm, dashdotted,,mark options = {scale= 1.0, solid}] 
coordinates{(25.0, 13.7241)( 81.0 , 13.7158 )( 289.0 , 11.6548 )( 352.0 , 7.96452 )( 463.0 , 5.47236 )( 646.0 , 3.90819 )( 1045.0 , 2.32517 )( 1686.0 , 1.45043 )( 2022.0 , 1.20937 )( 3641.0 , 0.80375 )( 7491.0 , 0.653383 )( 14363.0 , 0.57218 )( 39473.0 , 0.202575 )( 50843.0 , 0.108042 )( 63542.0 , 0.0672869 )( 87090.0 , 0.0451663 )( 117904.0 , 0.0333104 )( 156698.0 , 0.0266201 )( 210502.0 , 0.02221 )( 312177.0 , 0.0170385 )( 462254.0 , 0.0136912 )( 880881.0 , 0.00972972 )( 1401761.0 , 0.0074619 )};
\addlegendentry{$\eta$ (BJK)} 
\addplot[color=gold,  mark=oplus*, line width = 0.5mm, dashdotted,,mark options = {scale= 1.0, solid}]
coordinates{(25.0, 0.077891666488992)( 81.0 , 0.121846769481 )( 289.0 , 0.14133350854 )( 352.0 , 0.133275483089 )( 463.0 , 0.122409224994 )( 646.0 , 0.112346913957 )( 1045.0 , 0.0939670812613 )( 1686.0 , 0.0695044661202 )( 2022.0 , 0.0631634951785 )( 3641.0 , 0.0387885711783 )( 7491.0 , 0.0271141533233 )( 14424.0 , 0.0196417459922 )( 39409.0 , 0.0111634242719 )( 49670.0 , 0.0097883590876 )( 61741.0 , 0.00938542363591 )( 83693.0 , 0.00997187428367 )( 111831.0 , 0.00816352050847 )( 150840.0 , 0.00617719139133 )( 201913.0 , 0.00824075053185 )( 269999.0 , 0.0069495957833 )( 340794.0 , 0.0055849030632 )( 429740.0 , 0.00444119576172 )( 606132.0 , 0.00654528218349 )( 845005.0 , 0.00566963230724 )( 1114438.0 , 0.00456825081873 )};
\addlegendentry{$\|u-u_h\|_a$ (Kuzmin)} 
\addplot[color=navy_blue,  mark=oplus*, line width = 0.5mm, dashdotted,,mark options = {scale= 1.0, solid}] 
coordinates{(25.0, 13.7241)( 81.0 , 13.7158 )( 289.0 , 11.6547 )( 352.0 , 7.96451 )( 463.0 , 5.47235 )( 646.0 , 3.90816 )( 1045.0 , 2.32506 )( 1686.0 , 1.44996 )( 2022.0 , 1.20874 )( 3641.0 , 0.802179 )( 7491.0 , 0.652303 )( 14424.0 , 0.57113 )( 39409.0 , 0.203022 )( 49670.0 , 0.111413 )( 61741.0 , 0.0717912 )( 83693.0 , 0.0467526 )( 111831.0 , 0.0356211 )( 150840.0 , 0.0279488 )( 201913.0 , 0.0247254 )( 269999.0 , 0.020554 )( 340794.0 , 0.0174715 )( 429740.0 , 0.0150476 )( 606132.0 , 0.0148084 )( 845005.0 , 0.0125027 )( 1114438.0 , 0.0104006 )};
\addlegendentry{$\eta$ (Kuzmin)} 
\addplot[color=violet,  line width = 0.5mm, dashdotted,,mark options = {scale= 1.0, solid}]
coordinates{(25.0, 0.2)( 81.0 , 0.111111111111 )( 289.0 , 0.0588235294118 )( 352.0 , 0.0533001790889 )( 463.0 , 0.046473941234 )( 646.0 , 0.0393444737682 )( 1045.0 , 0.0309344112445 )( 1686.0 , 0.0243540512072 )( 2022.0 , 0.0222387014401 )( 3641.0 , 0.0165725623088 )( 7491.0 , 0.0115539398286 )( 14363.0 , 0.00834406002752 )( 39473.0 , 0.00503326657868 )( 50843.0 , 0.00443490596619 )( 63542.0 , 0.00396706722404 )( 87090.0 , 0.00338856526498 )( 117904.0 , 0.00291229745312 )( 156698.0 , 0.00252620315551 )( 210502.0 , 0.00217957534588 )( 312177.0 , 0.00178977957922 )( 462254.0 , 0.00147082045495 )( 880881.0 , 0.00106547037463 )( 1401761.0 , 0.000844623213998 )};
\addlegendentry{Optimal rate $\mathcal{O}(h)$} 
\end{loglogaxis}
\end{tikzpicture}\hspace*{1em}\begin{tikzpicture}[scale=0.75]
\begin{loglogaxis}[
    legend pos=south west, xlabel = $\#\ \mathrm{dof}$,
    legend cell align ={left},, title = {$\varepsilon=10^{-3}$},
    legend style={nodes={scale=0.75, transform shape}}]
\addplot[color=red,  mark=square*, line width = 0.5mm, dashdotted,,mark options = {scale= 1.5, solid}]
coordinates{(25.0, 0.07858936443668)( 81.0 , 0.122148359896 )( 289.0 , 0.141635427858 )( 1089.0 , 0.133469285398 )( 4225.0 , 0.122404454653 )( 16641.0 , 0.111686213105 )( 66049.0 , 0.0921912406045 )( 263169.0 , 0.062084993431 )( 1050625.0 , 0.0364598809841 )};
\addlegendentry{$\|u-u_h\|_a$ (BJK+Uniform)} 
\addplot[color=blue,  mark=square*, line width = 0.5mm, dashdotted,,mark options = {scale= 1.5, solid}] 
coordinates{(25.0, 0.07858936443668)( 81.0 , 0.122148359896 )( 289.0 , 0.141635427858 )( 352.0 , 0.133417158202 )( 463.0 , 0.122492774489 )( 646.0 , 0.112382684179 )( 1045.0 , 0.0936672027211 )( 1686.0 , 0.0692271478092 )( 2022.0 , 0.0628558741378 )( 3641.0 , 0.0392768611604 )( 7491.0 , 0.0265327634077 )( 14363.0 , 0.0194653971443 )( 39473.0 , 0.0108185060683 )( 50843.0 , 0.0095048163167 )( 63542.0 , 0.00921857759092 )( 87090.0 , 0.00790197230352 )( 117904.0 , 0.00592231489241 )( 156698.0 , 0.00506278044293 )( 210502.0 , 0.0042277082571 )( 312177.0 , 0.00323738248665 )( 462254.0 , 0.00264023149239 )( 880881.0 , 0.00188896077616 )( 1401761.0 , 0.0014409911019 )};
\addlegendentry{$\|u-u_h\|_a$ (BJK+Adaptive)} 
\addplot[color=gold,  mark=oplus*, line width = 0.5mm, dashdotted,,mark options = {scale= 1.5, solid}]
coordinates{(25.0, 0.077891666488992)( 81.0 , 0.121846769481 )( 289.0 , 0.14133350854 )( 1089.0 , 0.133267362871 )( 4225.0 , 0.122336162121 )( 16641.0 , 0.11173100268 )( 66049.0 , 0.0925789492918 )( 263169.0 , 0.0624279199852 )( 1050625.0 , 0.0355373089066 )};
\addlegendentry{$\|u-u_h\|_a$ (Kuzmin+Uniform)} 
\addplot[color=navy_blue,  mark=oplus*, line width = 0.5mm, dashdotted,,mark options = {scale= 1.5, solid}] 
coordinates{(25.0, 0.077891666488992)( 81.0 , 0.121846769481 )( 289.0 , 0.14133350854 )( 352.0 , 0.133275483089 )( 463.0 , 0.122409224994 )( 646.0 , 0.112346913957 )( 1045.0 , 0.0939670812613 )( 1686.0 , 0.0695044661202 )( 2022.0 , 0.0631634951785 )( 3641.0 , 0.0387885711783 )( 7491.0 , 0.0271141533233 )( 14424.0 , 0.0196417459922 )( 39409.0 , 0.0111634242719 )( 49670.0 , 0.0097883590876 )( 61741.0 , 0.00938542363591 )( 83693.0 , 0.00997187428367 )( 111831.0 , 0.00816352050847 )( 150840.0 , 0.00617719139133 )( 201913.0 , 0.00824075053185 )( 269999.0 , 0.0069495957833 )( 340794.0 , 0.0055849030632 )( 429740.0 , 0.00444119576172 )( 606132.0 , 0.00654528218349 )( 845005.0 , 0.00566963230724 )( 1114438.0 , 0.00456825081873 )};
\addlegendentry{$\|u-u_h\|_a$ (Kuzmin+Adaptive)} 
\addplot[color=violet,  line width = 0.5mm, dashdotted,,mark options = {scale= 1.5, solid}]
coordinates{(25.0, 0.2)( 81.0 , 0.111111111111 )( 289.0 , 0.0588235294118 )( 1089.0 , 0.030303030303 )( 4225.0 , 0.0153846153846 )( 16641.0 , 0.0077519379845 )( 66049.0 , 0.00389105058366 )( 263169.0 , 0.00194931773879 )( 1050625.0 , 0.000975609756098 )};
\addlegendentry{Optimal rate $\mathcal{O}(h)$} 
\end{loglogaxis}
\end{tikzpicture}}
\caption{Example~\ref{ex:known_2d_boundary}: Error in energy norm with \afcse technique defined in Sec.~\ref{sec:afc_supg_est}.}\label{fig:error_comparison_example_2d_sol_boundary_layer_afcse}
\end{figure}
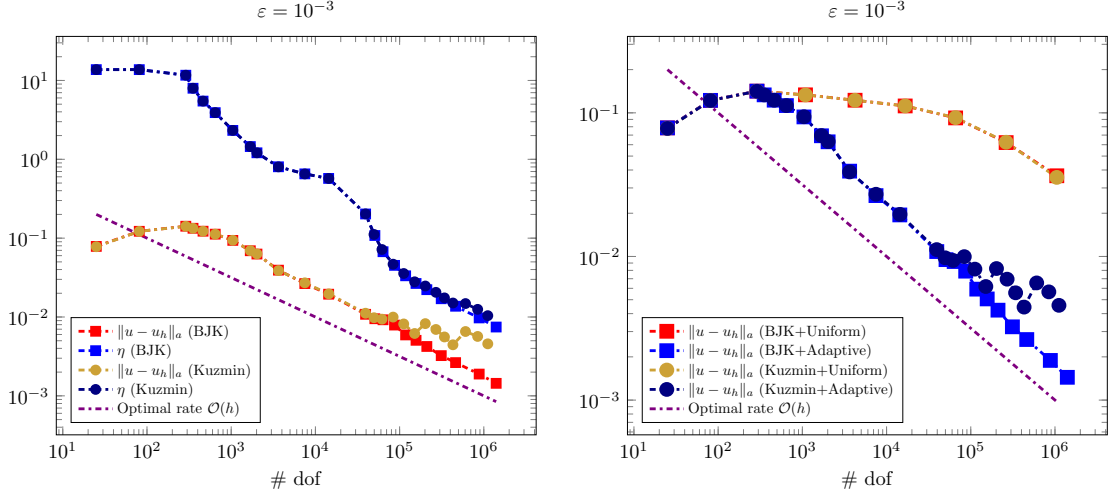


\begin{figure}[t!]
\centerline{\includegraphics[width=0.48\textwidth]{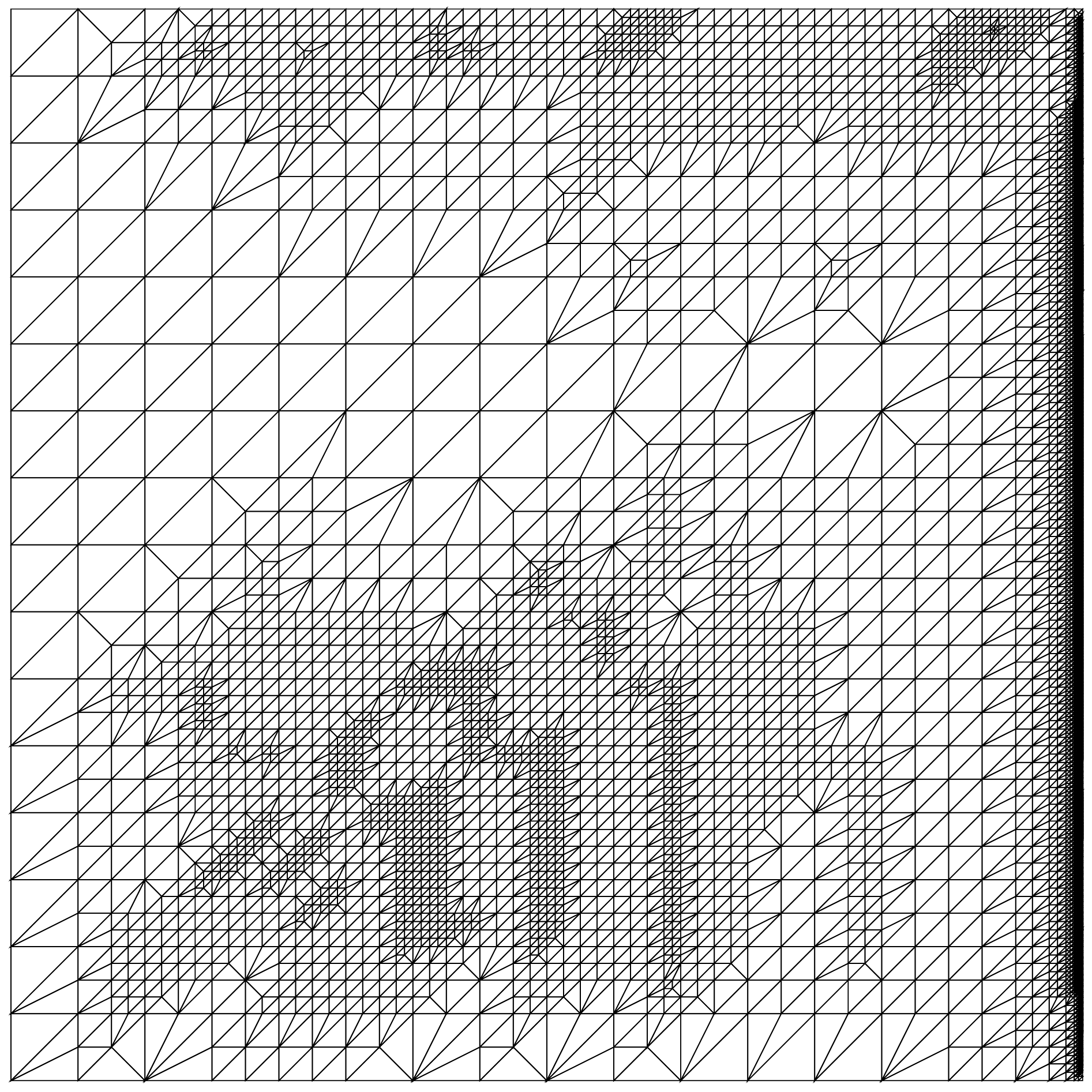}\hspace*{1em}
\includegraphics[width=0.48\textwidth]{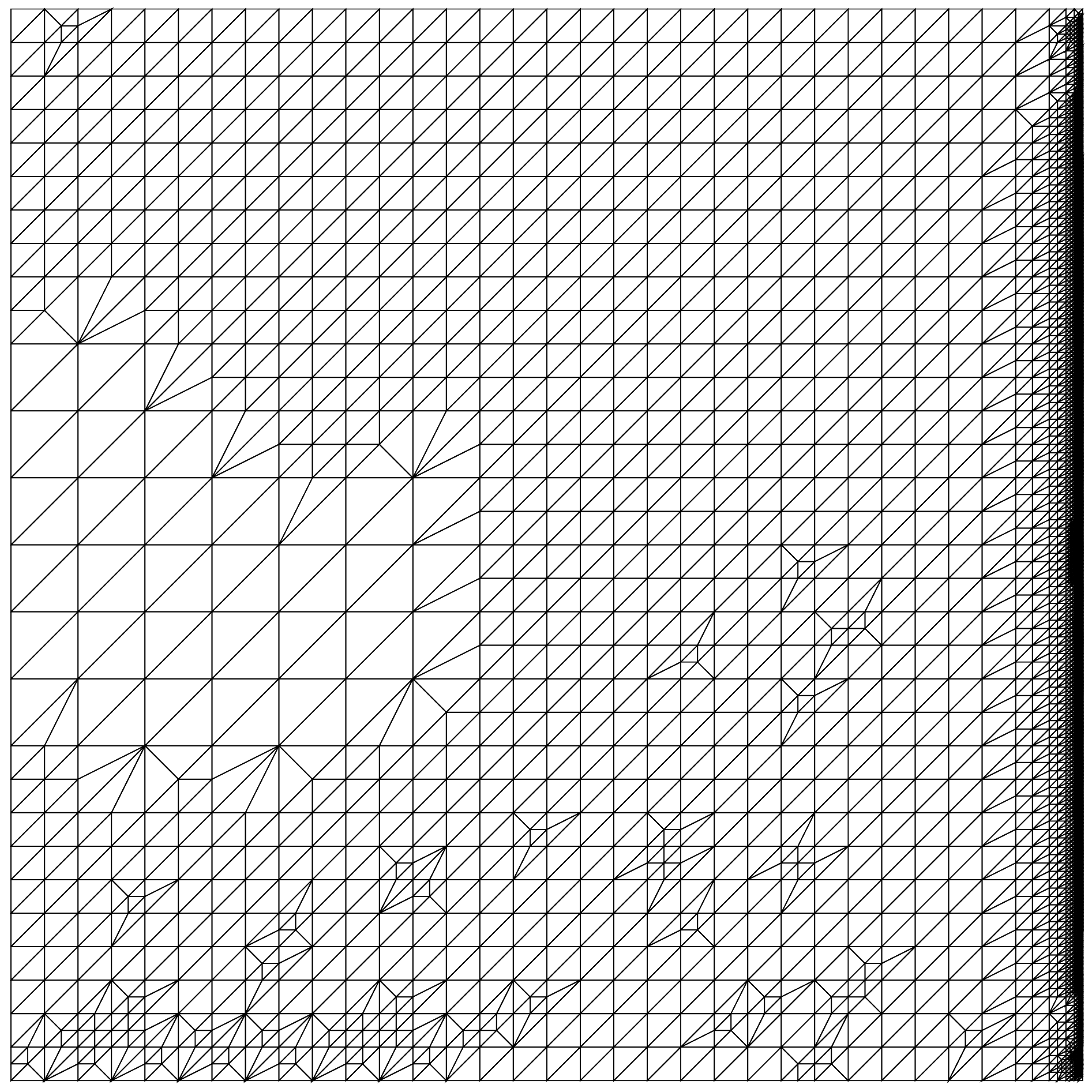}}
\caption{Example~\ref{ex:known_2d_boundary}: $14^{\mathrm{th}}$ adaptively refined grid with \afce technique. Kuzmin limiter ($\#\mathrm{dof}=22962$) (left) and BJK limiter ($\#\mathrm{dof}=23572$)(right).}\label{fig:adaptive_grids_example_1_ABR17}
\end{figure}

\subsection{Example with Interior and Boundary Layer}\label{ex:hmm86_post}
This example is proposed in \cite{HMM86}. It is given in $\Omega = (0,1)^2$ with 
$\bb = (\cos(-\pi/3),$ $ \sin(-\pi/3))$, $c=f=g=0$ and the Dirichlet boundary condition
$$
u_D = \begin{cases} 1 & (y=1\wedge x>0) \mbox{ or } (x=0 \wedge y>0.7),\\
0 & \mbox{else}.
\end{cases}
$$
Here, $\ep = 10^{-4}$ is considered. It is known that the solution exhibits an internal layer in the direction of the convection starting from the jump of the boundary condition at the left boundary and two exponential layers at the right and the lower boundary (see Fig.~\ref{fig:hmm86_sol_post}). A known solution to this problem is not available but we know that $u\in [0,1]$. 
This example serves for studying the adaptive grid refinement in the presence of different kinds of layers.

\begin{figure}[t!]
\centerline{\includegraphics[width=0.48\textwidth]{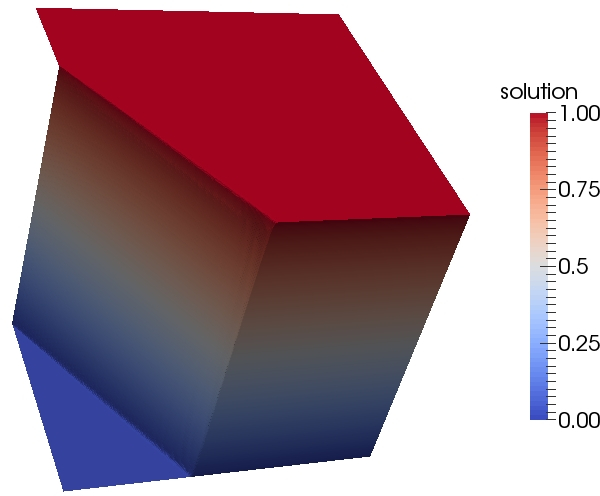}}
\caption{Example~\ref{ex:hmm86_post}. Solution (computed with the BJK limiter, level~9).}
\label{fig:hmm86_sol_post}
\end{figure}

An initial mesh was defined similar to the previous example, i.e., with two triangles by joining the points $(0,0)$ and $(1,1)$. The simulations were started with a level 2 grid (i.e., $\# \mathrm{dof}=25$), uniform refinement was performed till level 4 (i.e., $\# \mathrm{dof}=289$) and then the adaptive grid refinement was started. For this example, we do not have the presence of regions where the problem becomes locally diffusion-dominated because the refinement does not make the grid sufficiently fine for the considered diffusion parameter.

The $14^{\mathrm{th}}$ adaptively refined grids with conforming closure and \afce technique are shown in Fig.~\ref{fig:hmm86_conf_grid_jj20} for the Kuzmin limiter (left) and the BJK limiter (right), respectively. Here we see that we have the presence of non-Delaunay triangulation but we could note that the DMP was satisfied for both the limiters. This result shows that using the Kuzmin limiter might lead to solutions that satisfy the DMP even if an essential assumption of the analysis (Delaunay triangluation \cite[Remark~14]{BJK16}) is not satisfied. Comparing the refinement for both the limiters, we observe that the number of mesh cells is comparable for both the limiters (see Fig.~\ref{fig:hmm86_conf_grid_jj20} for $\#\mathrm{dof}$).

\begin{figure}[t!]
\centerline{\includegraphics[width=0.48\textwidth]{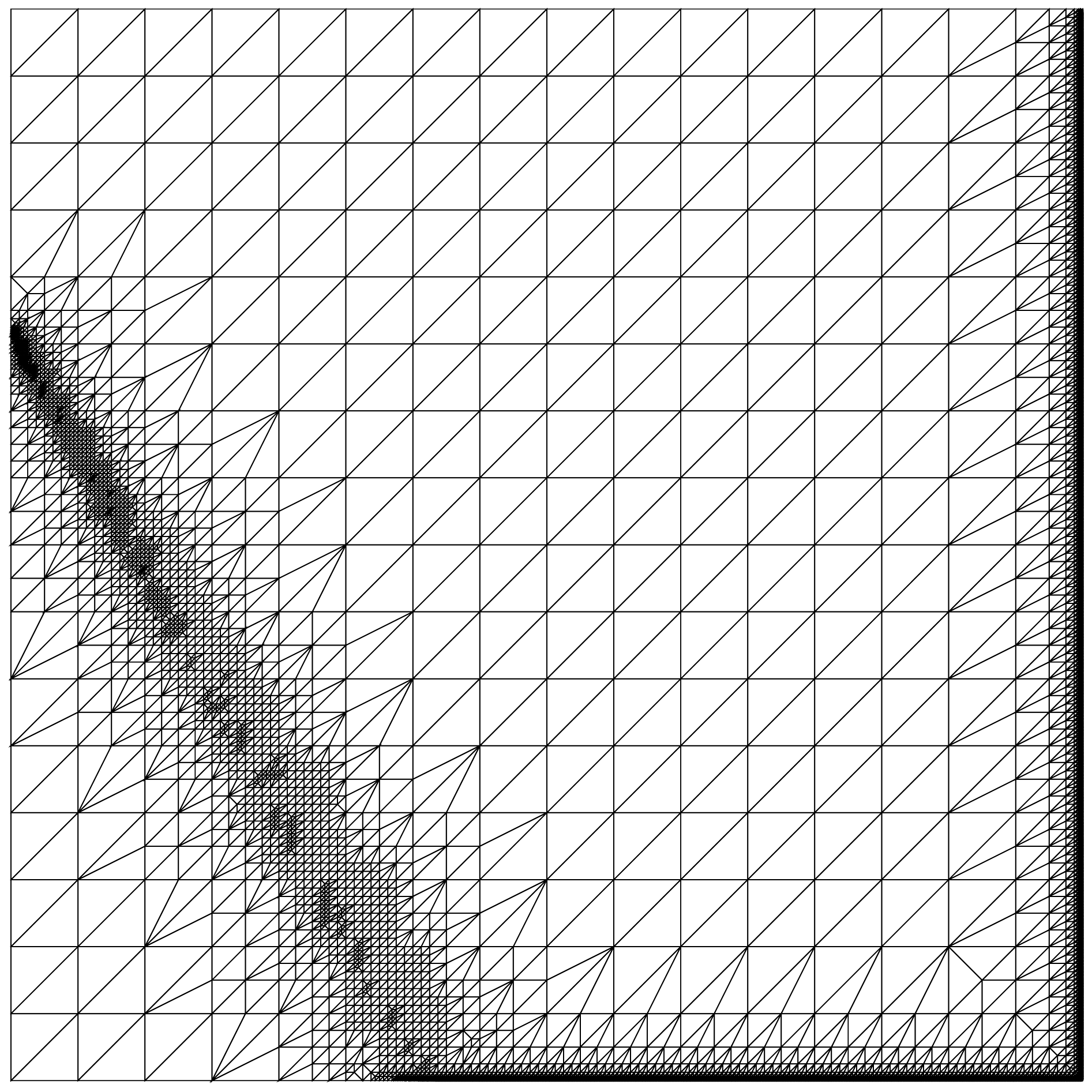}
\includegraphics[width=0.48\textwidth]{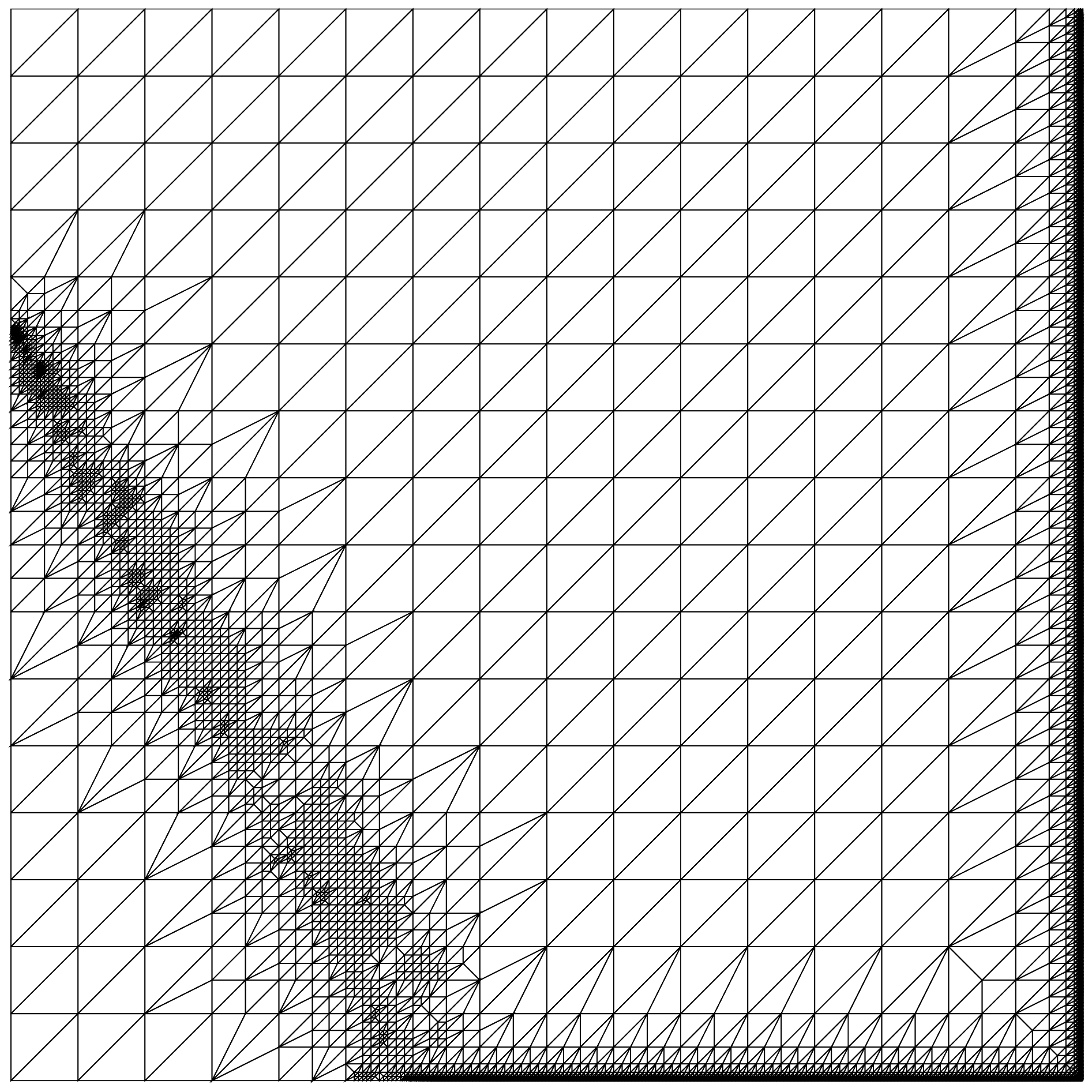}}

\caption{Example~\ref{ex:hmm86_post}: $14^{\mathrm{th}}$ adaptively refined grid with \afce technique and with conforming closure. Kuzmin limiter ($14^{\mathrm{th}}$ grid: $\#\mathrm{dof}=28548$ (left) and BJK limiter ($14^{\mathrm{th}}$ grid: $\#\mathrm{dof}=28120$) (right).}
\label{fig:hmm86_conf_grid_jj20}
\end{figure}

Next, we study the adaptive grid refinement for the \afcse technique. The $14^{\mathrm{th}}$ adaptively refine grids with conforming closure are shown in Fig.~\ref{fig:hmm86_conf_grid_supg} for the Kuzmin limiter (left) and the BJK limiter (right), respectively. Here we observe that the mesh cells near the internal layer are not refined that much as compared to the \afce technique. Also, we see that the limiters do not play an important role in the adaptive refinement. To be precise, the $\#\mathrm{dof}$ are comparable for both the limiters and the meshes look much more similar than in Fig.~\ref{fig:hmm86_conf_grid_jj20}.

\begin{figure}[t!]
\centerline{
\includegraphics[width=0.48\textwidth]{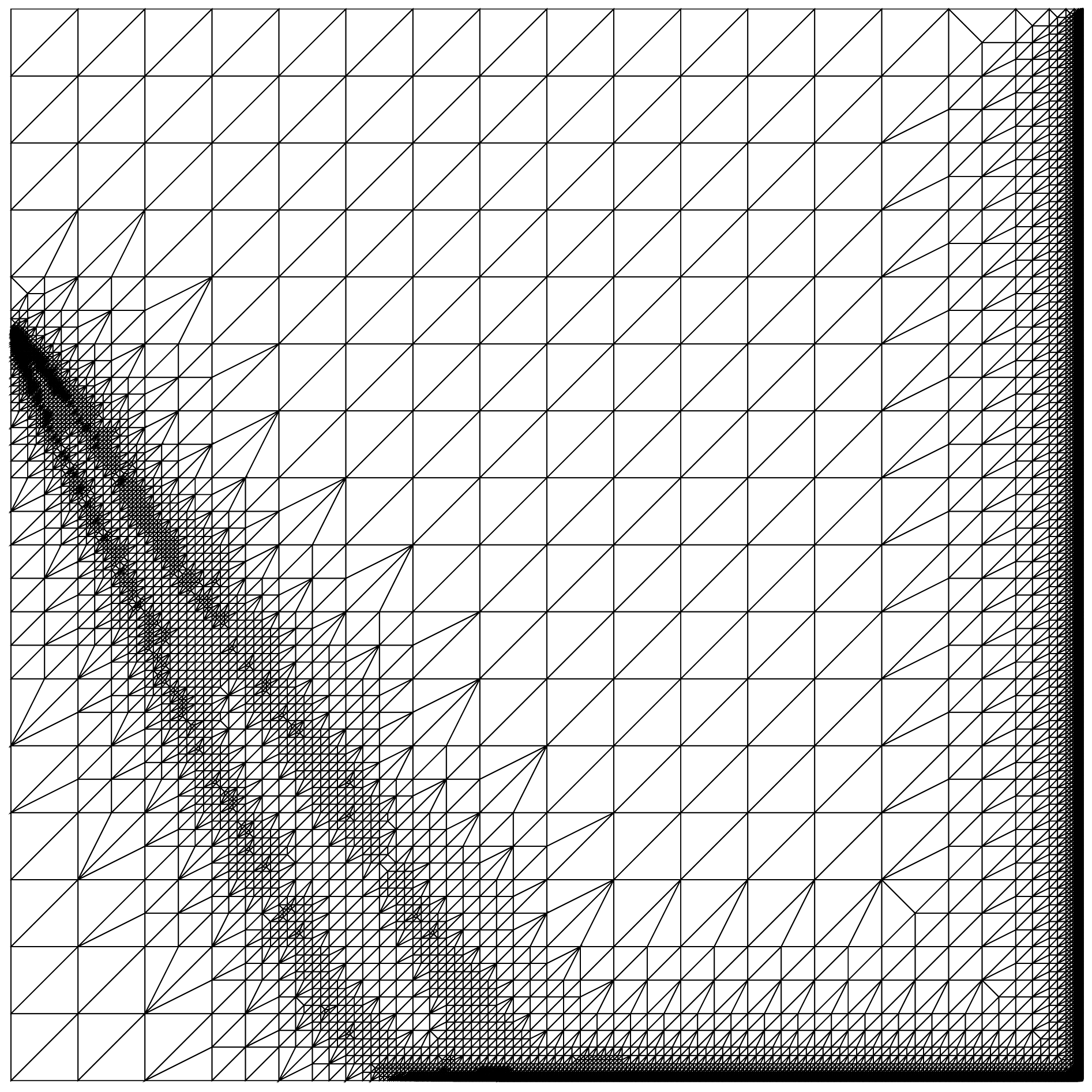}
\includegraphics[width=0.48\textwidth]{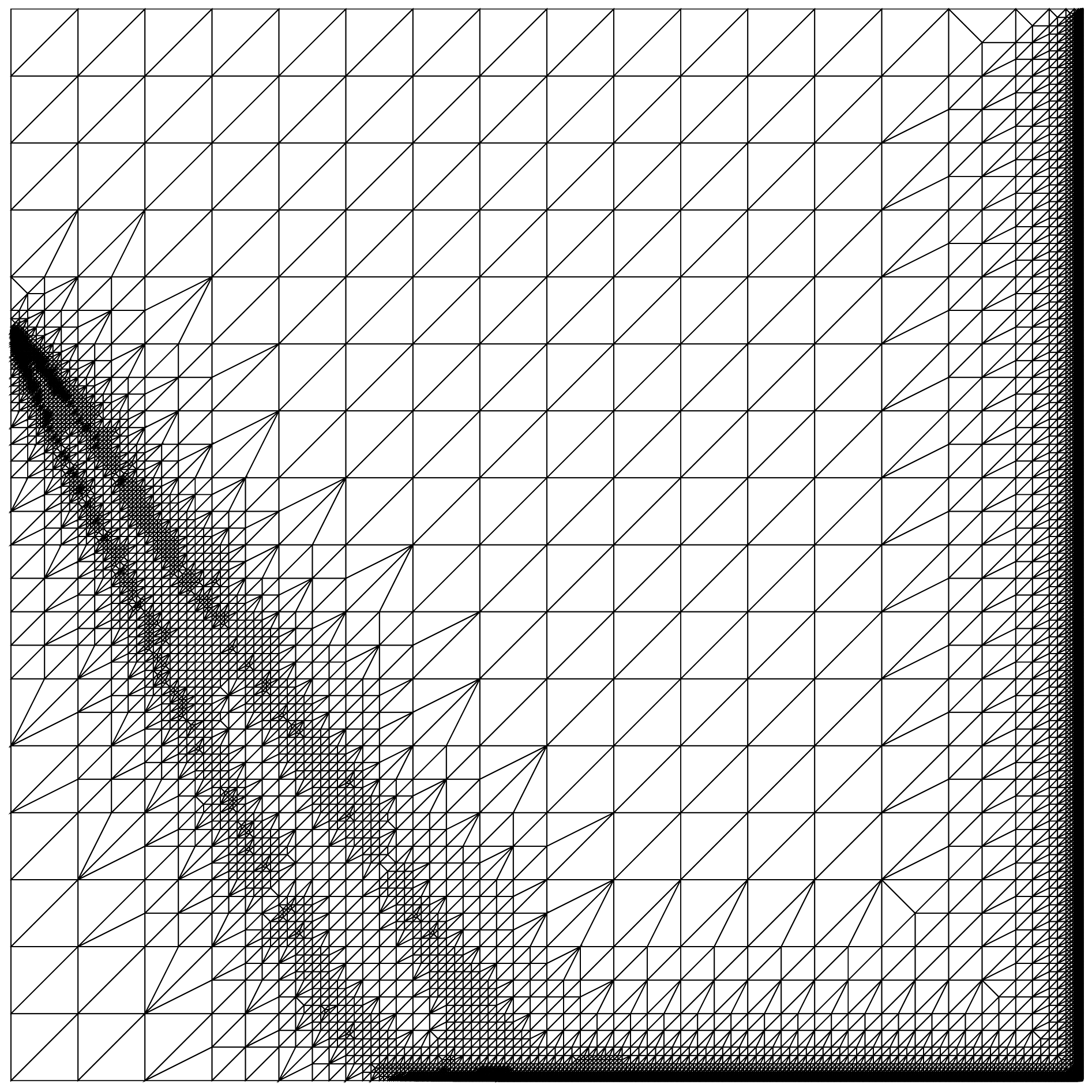}}

\caption{Example~\ref{ex:hmm86_post}: $14^{\mathrm{th}}$ adaptively refined grid with \afcse technique and with conforming closure. Kuzmin limiter ($14^{\mathrm{th}}$ grid: $\#\mathrm{dof}=100620$ (left) and BJK limiter ($14^{\mathrm{th}}$ grid: $\#\mathrm{dof}=100538$) (right).}
\label{fig:hmm86_conf_grid_supg}
\end{figure}

To check the thickness of the interior layer we follow the idea described in \cite[Eq.~(48)]{JK07_1}. We define
\begin{equation}\label{eq:smear}
smear_{\mathrm{int}}= x_2-x_1,
\end{equation}
where $x_1$ is the $x-$coordinate of the first point on the cut line $(x,0.25)$ with $u_h(x_1,0.25)\geq 0.1$ and $x_2$ is the $x-$coordinate of the first point with $u_h(x_1,0.25)\geq 0.9$. We note that in Fig.~\ref{fig:hmm86_thickness_layer}, the layers are most properly resolved for \afce technique as compared to the \afcse technique irrespective of the choice of limiters. 

\revi{One can observe in Fig.~\ref{fig:example_1_ABR17_supg_eta_afc_eta_supg_comp} that the \afcse estimator is dominated by $\eta_{\mathrm{SUPG}}$. If we look at the results from \cite{JN13}, where this estimator is proposed, specifically for \cite[Example~3]{JN13}, where the example has different kinds of layers (an exponential layer at the circle and parabolic (weaker) layers after the circle); it has been noted that the SUPG estimator refines mostly the strongest singularities, which for that example, is the exponential layer, and the weaker layers are not refined properly. This can be observed in Fig.~\ref{fig:hmm86_conf_grid_supg} as well; the SUPG estimator has
problems refining the parts of the grid with weaker singularities since the most effort of refinement goes in the stronger exponential layers at the boundary.}

Overall, for adaptive grid refinement, the \afce technique does a much better job since all layers are refined properly, not only the strongest layer. 

\begin{figure}[t!]
\centerline{\begin{tikzpicture}[scale=0.75]
\begin{loglogaxis}[
    legend pos=north east, xlabel = $\#\ \mathrm{dof}$, ylabel = $smear_{\mathrm{int}}$,
    legend cell align ={left},, title = {$\varepsilon=10^{-4}$},
    legend style={nodes={scale=0.75, transform shape}}]
\addplot[color=red,  mark=oplus*, line width = 0.5mm,dashdotted,, mark options = {scale= 1.5, solid}] 
coordinates{( 289.0 , 0.2026 )( 385.0 , 0.2026 )( 566.0 , 0.2026 )( 682.0 , 0.183 )( 962.0 , 0.183 )( 1246.0 , 0.1877 )( 1818.0 , 0.1763 )( 2273.0 , 0.1377 )( 3483.0 , 0.1081 )( 4754.0 , 0.0999 )( 6996.0 , 0.0814 )( 9554.0 , 0.0645 )( 14030.0 , 0.0572 )( 19325.0 , 0.0447 )( 28548.0 , 0.039 )( 39588.0 , 0.0351 )( 62033.0 , 0.0323 )( 91310.0 , 0.0316 )( 121576.0 , 0.031 )( 180498.0 , 0.0304 )( 219613.0 , 0.0305 )( 261517.0 , 0.0304 )( 334521.0 , 0.0305 )( 403679.0 , 0.0308 )( 464229.0 , 0.0306 )( 555817.0 , 0.0309 )( 766861.0 , 0.0308 )};
\addlegendentry{Kuzmin limiter (\afce technique)} 
\addplot[color=blue,  mark=oplus*, line width = 0.5mm,dashdotted,, mark options = {scale= 1.5, solid}] 
coordinates{( 289.0 , 0.2026 )( 397.0 , 0.1881 )( 579.0 , 0.1825 )( 879.0 , 0.1839 )( 1129.0 , 0.1829 )( 1772.0 , 0.1811 )( 2160.0 , 0.1836 )( 3703.0 , 0.1866 )( 7253.0 , 0.1122 )( 11142.0 , 0.105 )( 15890.0 , 0.1096 )( 32059.0 , 0.1002 )( 44438.0 , 0.0681 )( 66687.0 , 0.0666 )( 100620.0 , 0.0605 )( 150214.0 , 0.0515 )( 222308.0 , 0.0416 )( 295495.0 , 0.032 )( 416521.0 , 0.0327 )( 675783.0 , 0.0321 )( 941958.0 , 0.0322 )( 1392563.0 , 0.031 )};
\addlegendentry{Kuzmin limiter (\afcse technique)} 
\end{loglogaxis}
\end{tikzpicture}\hspace*{1em}
\begin{tikzpicture}[scale=0.75]
\begin{loglogaxis}[
    legend pos=north east, xlabel = $\#\ \mathrm{dof}$, ylabel = $smear_{\mathrm{int}}$,
    legend cell align ={left},, title = {$\varepsilon=10^{-4}$},
    legend style={nodes={scale=0.75, transform shape}}]
\addplot[color=red,  mark=oplus*, line width = 0.5mm,dashdotted,, mark options = {scale= 1.5, solid}] 
coordinates{( 289.0 , 0.1544 )( 385.0 , 0.1534 )( 452.0 , 0.1544 )( 691.0 , 0.131 )( 968.0 , 0.153 )( 1217.0 , 0.1475 )( 1775.0 , 0.1476 )( 2290.0 , 0.1396 )( 3419.0 , 0.1189 )( 4595.0 , 0.0888 )( 6895.0 , 0.0624 )( 9261.0 , 0.0606 )( 13650.0 , 0.0437 )( 18957.0 , 0.0396 )( 28120.0 , 0.0332 )( 38194.0 , 0.032 )( 59730.0 , 0.0303 )( 85976.0 , 0.0304 )( 98293.0 , 0.0304 )( 118413.0 , 0.0302 )( 147025.0 , 0.03 )( 183131.0 , 0.0301 )( 236401.0 , 0.0301 )( 308290.0 , 0.0302 )( 415735.0 , 0.0302 )};
\addlegendentry{BJK limiter (\afce technique)} 
\addplot[color=blue,  mark=oplus*, line width = 0.5mm,dashdotted,, mark options = {scale= 1.5, solid}] 
coordinates{( 289.0 , 0.1544 )( 397.0 , 0.1383 )( 583.0 , 0.1284 )( 879.0 , 0.1304 )( 1129.0 , 0.1234 )( 1774.0 , 0.1364 )( 2160.0 , 0.1269 )( 3703.0 , 0.1371 )( 7263.0 , 0.0808 )( 11138.0 , 0.079 )( 15900.0 , 0.08 )( 32083.0 , 0.071 )( 44444.0 , 0.0415 )( 66573.0 , 0.0456 )( 100538.0 , 0.0442 )( 150083.0 , 0.0408 )( 219751.0 , 0.0303 )( 299195.0 , 0.0301 )( 416711.0 , 0.0301 )( 654933.0 , 0.0301 )( 922309.0 , 0.0301 )};
\addlegendentry{BJK limiter (\afcse technique)} 
\end{loglogaxis}
\end{tikzpicture}}
\caption{Example~\ref{ex:hmm86_post}: Thickness of interior layer. Kuzmin limiter (left), BJK limiter (right).}
\label{fig:hmm86_thickness_layer}
\end{figure}
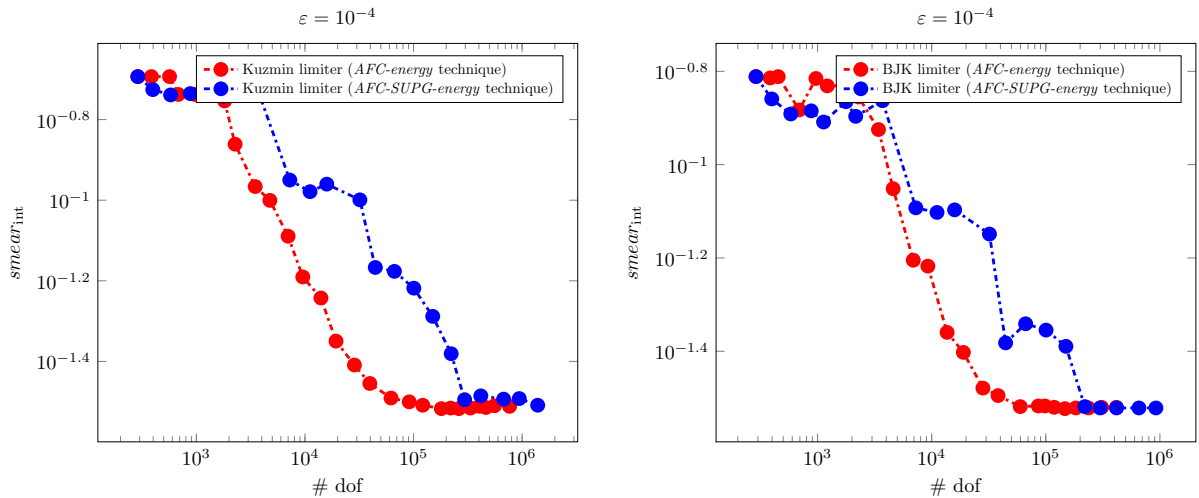

\section{Summary}
In this work, a new residual-based a posteriori error estimator has been derived in the energy norm for AFC schemes (\afce). Another approach for finding an upper bound in a posteriori way using the SUPG solution (\afcse) has also been discussed.

The following conclusions can be made from the numerical simulations. 
\begin{enumerate}
\item The effectivity index of the error estimator with \afce was not robust with respect to $\varepsilon$. The effectivity index was quite large for a strongly convection-dominated case, which eventually decreased as the mesh became finer. 
\item For the \afcse technique, the effectivity index was better than the \afce technique. 
\item The choice of limiter did not play an important role in \afcse technique as the dominating term was $\eta_{\mathrm{SUPG}}$. Because of this dominating nature, one gets very similar refined grids and effectivity indices for both the limiters.
\item For the Kuzmin limiter and the \afce technique, a reduced order of convergence can be observed with conforming closure using red-green refinements as adaptive refinement leads to locally diffusion-dominated problems. This kind of reduction of the order of convergence is not observed with the BJK limiter.
\item The AFC contribution $\eta_{d_h}$ is the dominating term in the estimator $\eta$ for the Kuzmin limiter. In contrast, it is the dominating term for the BJK limiter in the convection-dominated situation, but if the layer becomes to be resolved, then no longer.
\item With adaptive grid refinement, the problem could become locally diffusion-dominated. Then, one has to use the BJK limiter because, with the Kuzmin limiter, the error may become non-convergent. This situation might only happen if the diffusion coefficient is comparably large with respect to the mesh size.
\item For a small diffusion coefficient, one does not run into the previous point's issues. One has to use the Kuzmin limiter because of the difficulties encountered while solving the nonlinear problems with the BJK limiter, see \cite{JJ19}.
\item  For adaptive grid refinement and problems with different layers, the \afce technique refines the grid much better than the \afcse technique.
\end{enumerate}

In summary, the  \afcse technique gave better results than the \afce technique with respect to the effectivity index. In contrast, the \afce technique gave better results with adaptive grid refinement. For convection-dominated problems, the BJK limiter gave a better effectivity index as compared to the Kuzmin limiter. Still, difficulties arise in solving the nonlinear problem associated with the BJK limiter for a small diffusion. Future work of the research relates to the estimator's behavior on grids with hanging nodes, development of robust estimators, numerical studies in 3d, and extending the analysis for the local lower bound.
\section{Acknowledgements}
The work of the author has been supported by Berlin Mathematical School (BMS). The author would like to thank Prof. Dr. Volker John for many fruitful discussions and suggestions.
\bibliographystyle{alpha}
\bibliography{afc_apost}
\end{document}